\documentclass{article}
\usepackage[utf8]{inputenc}
\usepackage{amsmath}
\usepackage{amsthm}
\usepackage{amssymb}
\usepackage{amsfonts}
\usepackage{authblk}
\usepackage{mathtools}
\usepackage{stmaryrd}
\usepackage{tikz-cd}
\usepackage{enumerate}
\usepackage{hyperref} \hypersetup{colorlinks=true,citecolor=[rgb]{0.153,0.255,0.545},urlcolor=black,colorlinks=true,linkcolor=[rgb]{0.153,0.255,0.545}}
\usepackage{bbm}
\usepackage{multirow}
\usepackage{multicol}
\usepackage{pdflscape}
\usepackage{comment}
\usepackage[all,2cell]{xy}
\UseTwocells

\usepackage{todonotes}

\usepackage[top=1in, bottom=1.25in, left=1.25in, right=1.25in]{geometry}

\newtheorem{theorem}{Theorem}[section]
\newtheorem{introthm}{Theorem}
\newtheorem{introconjecture}[introthm]{Conjecture}

\newtheorem*{theorem*}{Theorem}
\newtheorem{proposition}[theorem]{Proposition}
\newtheorem{lemma}[theorem]{Lemma}
\newtheorem{corollary}[theorem]{Corollary}
\newtheorem{conjecture}[theorem]{Conjecture}
\theoremstyle{definition}
\newtheorem{example}[theorem]{Example}
\newtheorem{remark}[theorem]{Remark}

\newtheorem{definition}[theorem]{Definition}

\newcommand{\id}{\textup{id}}

\newcommand{\GL}{\textup{GL}}
\newcommand{\SL}{\textup{SL}}

\newcommand{\CC}{\mathbb{C}}
\renewcommand{\AA}{\mathbb{A}}
\newcommand{\Stck}{\textup{\bf{}Stck}}

\newcommand{\K}{\textup{K}}
\newcommand{\GG}{\mathbb{G}}
\newcommand{\FF}{\mathbb{F}}
\newcommand{\PP}{\mathbb{P}}
\newcommand{\LL}{q}
\newcommand{\ZZ}{\mathbb{Z}}
\newcommand{\QQ}{\mathbb{Q}}
\newcommand{\fH}{\mathcal{H}}
\newcommand{\fV}{\mathcal{V}}
\newcommand{\Var}{\textup{\bf{}Var}}

\newcommand{\Hom}{\textup{Hom}}

\newcommand{\U}{\mathbb{U}}

\renewcommand\S[1]{\mathrm{S}#1}

\DeclareMathOperator{\Rep}{Rep}
\DeclareMathOperator{\Sym}{Sym}
\DeclareMathOperator{\Res}{Res}
\DeclareMathOperator{\Ind}{Ind}
\DeclareMathOperator{\Conf}{Conf}
\DeclareMathOperator{\MHS}{MHS}
\DeclareMathOperator{\Irr}{Irr}

\title{\bf{}Motivic (Representation) Stability of Representation Varieties and Character Stacks}
\author[$\ddagger$]{Márton Hablicsek}
\author[$ $]{Jesse Vogel}
\affil[$\ddagger$]{\footnotesize Department of Mathematics, Leiden University, Niels Bohrweg 1, 2333 CA Leiden, Netherlands, hablicsekhm@math.leidenuniv.nl}
\date{\today}

\begin{document}

\maketitle

\begin{abstract}
    In this paper, we introduce the notions of motivic representation stability that is an algebraic counterpart of the notion of representation stability. In the process, we also introduce the notion of motivic decomposition for varieties equipped with an action of a finite group $G$. This motivic decomposition decomposes the virtual class of the variety with respect to irreducible rational representations of $G$.
    
    We also formulate conjectures on motivic representation stability in the context of representation varieties and character stacks, and we verify the conjectures for groups whose virtual classes have been extensively studied.
\\

    \noindent \textbf{Keywords:} representation stability, Grothendieck ring of varieties, representation variety, character stack.
    
    \noindent \textbf{AMS Subject Classification:} 14F45, 14M35, 20C05.
\end{abstract}
\section{Introduction} 

The purpose of this paper is twofold. On one hand, we provide a framework and computational tools to explore the concept of motivic stability and motivic representation stability, focusing on representation varieties of surface groups, free groups, and free abelian groups. Motivic representation stability extends the notion of representation stability \cite{church2013representation} by analyzing virtual classes of varieties in Grothendieck rings, offering an algebraic counterpart to representation stability that simultaneously captures representation stability in various cohomology theories.

On the other hand, we expand Chapter 7 of \cite{vogel2024motivicphd} and formulate conjectures regarding motivic stability and motivic representation stability in the context of representation varieties and character stacks over an algebraically closed field $k$ of characteristic 0. Specifically, we analyze the following cases.
\begin{itemize}
    \item Surface Groups, i.e., examining the motivic stability of $G$-representation varieties, $\Rep_G(M_g)$, of compact surfaces of increasing genus. In this case, we conjecture the following.
    
    \begin{introconjecture}\label{introcon:A}
    Let $G$ be a connected linear algebraic group over $k$. Let $\Rep_G(M_g)$ denote the $G$-representation variety of the surface group of a smooth compact genus $g$ surface. Then, 
    \begin{equation}\label{eq:limit}
        \lim_{g\to \infty}\frac{[\Rep_G(M_g)]}{[G^{2g}]}=\frac{[G/[G,G]]}{[G]}
    \end{equation}
    in the completed Grothendieck rings of stacks, $\widehat{\K_0(\Stck_k)}$.
    \end{introconjecture}
    \item Free Groups, i.e., examining the motivic representation stability of the representation varieties $G^n:=\Hom_{Gp}(F^n, G)$ corresponding to the free groups $F_n$. In this case, we conjecture that the character stack is motivically representation stable.
    \begin{introconjecture}\label{introcon:B}
    Let $G$ be a connected linear algebraic group over $k$. Then, the sequence of representation varieties, $G^n$, and the sequence of their corresponding character stacks, $[G^n/G]$, are motivically representation stable.
    \end{introconjecture}
    \item Free Abelian Groups, i.e., examining the motivic representation stability of the representation varieties $C_n(G):=\Hom_{Gp}(\ZZ^n, G)$. We conjecture that sequence $C_n(G)$ also satisfies motivic representation stability.
    \begin{introconjecture}\label{introcon:C}
    Let $G$ be a connected linear algebraic group over $k$, and assume that $G$ has connected center. Then, the sequence of representation varieties $C_n(G)$ and their corresponding character stacks are motivically representation stable.
    \end{introconjecture}
    \item Stability with variation of the rank for free Abelian groups: in this case, we conjecture that sequence $C_n(\GL_r(k))$ satisfies motivic stability as $r$ tends to infinity.
    \begin{introconjecture}\label{introcon:D}
    Fix a positive integer $n$. Then, the sequence of representation varieties $C_n(\GL_r(k))$ and their corresponding character stacks are motivically stable.
    \end{introconjecture}
\end{itemize}

We note that Conjectures \ref{introcon:C} and \ref{introcon:D} can be thought of the algebraic versions of Theorem 1.1 and Theorem 9.6 of \cite{ramras2021homological}.

Conjectures \ref{introcon:A}, \ref{introcon:B} and \ref{introcon:C} are verified in the paper in the cases of linear algebraic groups for which virtual classes of representation varieties have been studied and thus enough techniques have been developed \cite{thesisangel, habvog20, vogel2024motivic}. These groups include the groups of $\GL_r$, $\SL_r$ and the groups, $\U_r$, of upper triangular matrices of low ranks. Namely, we have the following results.
\begin{itemize}
    \item We prove Conjecture A in the cases of $\SL_2$ and $\U_r$ (for $r\leq 5)$.
    \item We prove Conjecture B in the cases of $\GL_r$ and $\U_r$ (for any rank), see Theorem \ref{thm:conjB}.
    \item We prove Conjecture C in the case of $\GL_r$ (for any rank), see Corollary \ref{cor:conjC}.
    \item We prove Conjecture D for $n=2$, see Theorem \ref{thm:motstabrank}.
\end{itemize}

The paper is organized as follows. In Section 2, we provide a framework for motivic representation stability. To achieve this goal, we study motivic decompositions with respect to finite group actions that is a decomposition of a $G$-virtual class with respect to the irreducible rational representation of a finite group $G$, see Theorem \ref{thm:motivicdecom}. We show a variant of the K\"unneth-formula and provide a comparison with mixed Hodge structures. In Section 3, we study Conjecture \ref{introcon:A} via point-counting methods and Topological Quantum Field theories. These methods are connected via natural transformations \cite{gonzalez2023arithmetic}. In Section 4, we investigate Conjectures \ref{introcon:B} and \ref{introcon:C} in the cases of $\GL_r$ and $\SL_r$. We also provide computational tools to find the virtual classes of representation varieties and character stacks corresponding to free or free Abelian groups. Using the motivic decomposition theorem, we explicitly compute the virtual classes of the varieties of commuting $n$-tuples in $\GL_2(k)$ and $\GL_3(k)$, see Theorems \ref{thm:commntuplesgl2} and \ref{thm:commntuplegl3}. Finally, in Section 5, we prove Conjecture \ref{introcon:D} in the case of $n=2$.
\section{Preliminaries}

In this section, we revisit the notion of motivic representation stability of \cite{vogel2024motivicphd}. The key tool used in this section is a motivic decomposition theorem with respect to rational representations.

\subsection{Motivic stability in the Grothendieck rings of $G$-varieties}

Let $S$ be a variety over an algebraically closed field $k$ of characteristic 0. In this paper, we will work with $G$-varieties, namely, varieties $X$ over $S$ equipped with an action of a linear algebraic group $G$ over $k$ so that 1) the map $X\to S$ is $G$-equivariant (with the trivial action of $G$ on $S$) and 2) $X$ can be covered by $G$-equivariant open affines. Morphisms of $G$-varieties are morphisms of varieties over $S$ that are $G$-equivariant. 

\begin{definition}
    The Grothendieck ring of $G$-varieties, $\K_0(\Var^G_S)$, is the free Abelian group generated by isomorphism classes of $G$-varieties over $S$ with modulo the relations of the form $[X]=[U]+[Z]$ where $Z$ is a $G$-invariant closed subvariety of $X$ and $U$ is the corresponding $G$-invariant open complement. 
\end{definition}

\begin{remark}
    Note that if $X$ and $Y$ are $G$-varieties over $S$, then $X\times_S Y$ is equipped with a natural $G$-action, namely, the diagonal $G$-action, making $\K_0(\Var^G_S)$ a ring.
\end{remark}

In the case where $G$ is the trivial group, we obtain the Grothendieck ring of varieties over $S$, that we will denote by $\K_0(\Var_S)$. For a $G$-variety $X$ over $k$, the class $[X]\in \K_0(\Var^G_k)$ is called the motivic (or virtual) class of $X$. We denote the virtual class of $\AA^1$ (with the trivial $G$-action) by $\LL$.

Let $G$ be a finite group and $H$ a subgroup of $G$. On the level of representations, we have induction and restriction functors. Corresponding to these functors, we have maps on the corresponding Grothendieck rings.

\begin{definition}
    Let $G$ be a finite algebraic group over $k$ and $H$ a subgroup of $G$. Let $S$ be a variety over $k$. We define the restriction functor
    \[\Res_H^G: \Var_S^G\to \Var_S^H\]
    as the functor that regards a $G$-variety an $H$-variety under the inclusion, and we define the induction functor
    \[\Ind_H^G: \Var_S^H\to \Var_S^G\]
    as the functor that maps an $H$-variety $X$ over $S$ to the $G$-variety $(G\times X)/H$ (here $H$ acts diagonally on $G\times X$). The resulting variety is indeed a $G$-variety with action given by multiplication on the factor of $G$.
\end{definition}


It is easy to see that these functors descend to the Grothendieck ring of varieties providing maps of Abelian groups
\[\Res_H^G: \K_0(\Var_S^G)\to \K_0(\Var_S^H)\quad\mbox{and}\quad \Ind_H^G: \K_0(\Var_S^H)\to \K_0(\Var_S^G).\]

A little bit more is true. The restriction map $\Res_H^G: \K_0(\Var_S^G)\to \K_0(\Var_S^H)$ is a ring homomorphism and it can be defined for any morphism of linear algebraic groups $H\to G$. However, for general linear algebraic groups, more care is needed in the case of the induction functor. In fact, we may face three kinds of problems: 1) the GIT quotient appearing in the induction functor may not exist, 2) the GIT quotient may not be a variety, and 3) the GIT quotient may not be motivic. Here, motivic means that if $G$ is an $S$-variety with an $H$-action, and $Z$ is an equivariant closed subvariety, then the identity holds
\[[\Ind_H^G(Z)]+[\Ind_H^G(X\setminus Z)]=[\Ind_H^G(X)]\]
in the Grothendieck ring $\K_0(\Var^G_S)$.

However, in the case when $H$ is a closed subgroup of a linear algebraic group $G$, all these issues are resolved. 

\begin{proposition}
    Let $G$ and $G'$ be linear algebraic groups, and $\rho: G'\to G$ be a homomorphism of algebraic groups over $k$. Let $S$ be a variety over $k$. Then, the restricting the action provides a functor
    \[\Res_{G'}^G: \Var_S^{G}\to \Var_S^{G'}\]
    that descends to a map of rings
    \[\Res_{G'}^G:\K_0(\Var_S^G)\to \K_0(\Var_S^{G'}).\]
    Furthermore, let $H$ be a closed subgroup of $G$. Then, the induction functor 
    \[\Ind_{H}^G: \Var_S^H\to \Var_S^G\]
    defined by sending an $H$-variety $X$ over $S$ to the variety $(G\times X)/H$ descends to a map of Abelian groups
    \[\Ind_H^G: \K_0(\Var_S^H)\to \K_0(\Var_S^G).\]
\end{proposition}

\begin{proof}
    We focus on the case of the induction map, the restriction map is easy.
    
    In the case when $H$ is a closed subgroup of $G$, the quotient $(G\times X)/H$ exists and it is a variety by \cite{Popov1994} (resolving problems 1) and 2)). Furthermore, if $Z$ is an $H$-equivariant closed subvariety of $X$, then $G\times Z$ is an $H$-equivariant closed subvariety of $G\times X$. Since the action of $H$ is free on $G\times X$, the GIT quotient is motivic (see, for instance, Theorem 4.2.11 of \cite{thesisangel}), meaning that
    \[[G\times Z/H]+[G\times(X\setminus Z)/H]=[G\times X/H].\]
    Finally, the $G\times Z/H$ (and $G\times(X\setminus Z)/H$) is a $G$-equivariant closed (and open) subvarieties of $G\times X/H$, we have that
    \[[\Ind_H^G(Z)]+[\Ind_H^G(X\setminus Z)]=[\Ind_H^G(X)]\]
in the Grothendieck ring $\K_0(\Var^G_S)$ proving our statement.
\end{proof}

\subsubsection{Motivic stability}

In this paper, we are concerned with families of varieties and their limiting virtual class. Explicitly, we consider the localization $M_\LL^G:=\K_0(\Var^G_k)[\LL^{-1}]$. This ring has a natural increasing filtration given by the powers of $\LL$:
\[0\subseteq \dots\subseteq F_nM_\LL^G\subseteq F_{n+1}M_\LL^G\subseteq\dots \subseteq M_\LL^G\]
where $F_nM_\LL^G$ is the subgroup generated by the elements of $M_\LL^G$ of the form $\frac{[X]}{\LL^s}$ where $X$ is an irreducible variety of dimension at most $s+n$. We denote the completion of $M_\LL^G$ with respect to this filtration by $\widehat{M_\LL^G}$.

The ring $\widehat{M_\LL^G}$ is equipped with a topology coming from the completion that allows us to consider limits of families of virtual classes as in \cite{vakil2015discriminants}.

\begin{definition}
    We say that a family of $G$-varieties $\{X_n\}_n$ is motivically stable if the limit
    \[\lim_{n\to \infty} \frac{[X_n]}{\LL^{\dim X_n}}\]
    exists in $\widehat{M_\LL^G}$.
\end{definition}

It is easy to see that the restriction and induction maps respect the filtrations, meaning that if $H$ is a closed subgroup of a linear algebraic group $G$ over $k$, then 
\[\Res_H^G(F_n M_\LL^G)\subseteq F_n M_\LL^H\]
and 
\[\Ind_H^G(F_n M_\LL^H)\subseteq F_{n+c}M_\LL^G\]
where $c$ is the codimension of $H$ in $G$. Therefore, we obtain the following.

\begin{corollary}\label{cor:resindmotivic}
    Let $H$ be a closed subgroup of a linear algebraic group $G$ over $k$. Then, the restriction and induction maps provide a continuous group homomorphism:
    \[\Res_H^G: \widehat{M_\LL^G}\to \widehat{M_\LL^H}\quad\mbox{and}\quad \Ind_H^G: \widehat{M_\LL^H}\to \widehat{M_\LL^G}.\]
    In particular, we have that
    \begin{itemize}
        \item if the family of $G$-varieties $\{X_n\}_n$ is motivically stable, then the family of $H$-varieties $\{\Res_H^G(X_n)\}_n$ is also motivically stable,
        \item if the family of $H$-varieties $\{Y_n\}_n$ is motivically stable, then the family of $G$-varieties $\{\Ind_H^G(Y_n)\}_n$ is also motivically stable.
    \end{itemize}
\end{corollary}

In the context of motivic stability, one of the most important sequence of varieties that has been studied is the sequence of symmetric powers $\{\Sym^n_G (X)\}_n$ of $G$-varieties: the $n$-th symmetric power of a variety $X$, $\Sym^n_G X:=X^n/S_n$, is naturally equipped with a $G$-variety structure induced by the diagonal action. 




The following lemma is key in order to establish motivic stability. It is a straightforward adaptation of Proposition 4.2 in \cite{vakil2015discriminants}.

\begin{lemma}[Proposition 7.1.12 of \cite{vogel2024motivicphd}]\label{lem:openstabilizeallstabilize}
    Let $X$ be a $G$-variety and $Z\subset X$ a closed $G$-invariant subvariety of small dimension: $\dim Z<\dim X$. Then, the symmetric powers of $X$ stabilize if and only if the symmetric powers of the open complement $U=X\setminus Z$ stabilize. Moreover, we have
    \[\lim_{n\to \infty}\frac{\Sym_G^n(X)}{\LL^{n\dim X}}=Z_G(Z, \LL^{-\dim X})\lim_{n\to \infty}\frac{\Sym_G^n(U)}{\LL^{n\dim U}}\]
    where $Z_G(Z,q^{-\dim X})$ denotes the motivic zeta function of the $G$-variety $X$ in the sense of \cite{kapranov2000elliptic}.
\end{lemma}

\subsubsection{Motivic stability in the Grothendieck ring of stacks}

In Section \ref{sec:surface} a slightly different motivic stability will be considered. For that, we consider Ekedahl's version of the Grothendieck ring of stacks defined as follows \cite{eke09}.

\begin{definition}
    The Grothendieck ring of stacks $\K_0(\Stck_k)$ is defined as the Abelian group generated by stacks of finite type over $k$ with affine stabilizers module the relations 1) $[\mathfrak{X}]=[\mathfrak{Z}]+[\mathfrak{U}]$ where $\mathfrak{Z}$ is a closed substack of $\mathfrak{X}$ with open complement $\mathfrak{U}$ and 2) the relations of the form 
    \[ [\mathfrak{E}] = [ \AA^n_k \times \mathfrak{X} ] \]
    for every vector bundle $\mathfrak{E} \to \mathfrak{X}$ of rank $n$. 
\end{definition}

Ekedahl shows that the Grothendieck ring of stacks over $k$ is isomorphic to the localization of the Grothendieck ring of varieties, $\K_0(\Var_{k})$, by inverting the class of the affine line $\LL$ and the classes of the form $\LL^n-1$.

We define motivic stability in this ring parallel to the case of the Grothendieck ring of varieties. Namely, we consider the natural increasing filtration given by the powers of the symbols $\LL^t-1$:
\[0\subseteq \dots\subseteq F_nK_0(\Stck_k)\subseteq F_{n+1}K_0(\Stck_k)\subseteq\dots \subseteq K_0(\Stck_k)\]
where $F_nK_0(\Stck_k)$ is the subgroup generated by the elements of $K_0(\Stck_k)$ of the form $\frac{[X]}{q^s\prod (\LL^i-1)^{s_i}}$ where $X$ is an irreducible variety, the product in the denominator is finite and $X$ is of dimension at most $n+s+\sum i\cdot s_i$. We denote the completion of $K_0(\Stck_k)$ with respect to this filtration by $\widehat{\K_0(\Stck_k)}$.

\begin{remark}\label{rem:epolymotmeasure}
    The E-polynomial is a motivic measure of $\K_0(\Var_\CC)$ meaning that it provides a ring homomorphism $\K_0(\Var_\CC)\to \ZZ[u,v]$. It sends a smooth and projective variety $X$ to its E-polynomial $E(X):=\sum_{i,j} h^{i,j}(X)u^iv^j\in \ZZ[u,v]$ (where $h^{i,j}(X)$ denotes the dimensions of the Hodge cohomology spaces $H^i(X, \Omega^j_X)$). This can be extended to any variety over $\CC$, in particular, this motivic measure sends the affine line to $E(\AA^1):=uv$. Using the above, we show that the E-polynomial can be extended to a motivic measure of $\widehat{\K_0(\Stck_\CC)}$. Indeed, consider the ring $\ZZ[u,v]$, and invert the element $u\cdot v$ in this ring. Then, we have a natural filtration given by the powers of $\frac{1}{uv}$ and one can consider the completion of the localized ring with respect to this filtration, $\widehat{\ZZ[u,v, \frac{1}{uv}]}$. Now, for the virtual class of a stack of the form $[V/\GL_n]$, we define its E-polynomial as 
    \[\frac{E(V)}{(uv)^{n^2}}\prod_{i=1}^n (1+\frac{1}{(uv)^i}+\frac{1}{(uv)^2i}+...)\]
    that is well-defined in $\widehat{\ZZ[u,v, \frac{1}{uv}]}$. Note, that the E-polynomial does not depend on the representation $[V/\GL_n]$ \cite{kres99, behdhi07}. We can see that the E-polynomial respects the filtration on $\K_0(\Stck_\CC)$, thus it provides a motivic measure $\widehat{\K_0(\Stck_\CC)}\to \widehat{\ZZ[u,v,\frac{1}{uv}]}$.
\end{remark}


\subsection{Motivic representation stability}

The goal of this section is to provide a framework in studying representation stability via virtual classes. Since, the proofs of the representation stability results \cite{ramras2021homological} for representation varieties, character varieties and character stacks do not depend on the actual cohomology theory, it is expected that representation stability holds on the motivic level. 

The difficulty in approaching this problem using the Grothendieck ring of varieties is twofold. First, it is not clear how to decompose a variety with an $S_n$-action into subvarieties corresponding to representations of $S_n$. Second, in order to do any kind of computations, one needs that this decomposition satisfies certain properties, for instance, K\"unneth-formuala, etc. 

These difficulties cannot be solved. To avoid the first problem, we have to make a choice on how to decompose a variety according to the representations of $S_n$. The choice is basically an appropriate set of subgroups in $S_n$. To deal with the second problem, we will show that, even though not every $S_n$-varieties can be used in K\"unneth-type formulas, there are still sufficient $S_n$-varieties that can be used for computational purposes.

We begin by addressing the first problem. The representations of $S_n$ are parametrized by Young tableaux: for a partition $\lambda$ of $n$, we denote the corresponding representation by $V_\lambda$. For a general partition $\lambda=(\lambda_1, \lambda_2, ...)$ (with $\lambda_1\geq \lambda_2\geq...$) and for an integer $n\geq |\lambda|+\lambda_1:=2\lambda_1+\lambda_2+...$, we denote the partition of $n$ given by $(n-|\lambda|,\lambda_1,\lambda_2,...)$ by $\lambda[n]$. In representation stability, one is interested in the stability of the representations of $V_{\lambda[n]}$.

Corresponding to a partition $\lambda=(\lambda_1, \lambda_2,...)$ of $n$, we have a subgroup
\[S_\lambda:=S_{\lambda_1}\times S_{\lambda_2}\times...\leq S_n.\]
Thus, for the sequence of representations $V_{\lambda[n]}$, we obtain a sequence of subgroups $S_{\lambda[n]}\leq S_n$. After setting up the notation, we arrive at the definition of motivic representation stability.

\begin{definition}\label{def:motrepstab}
Let $\{X_n\}_n$ be a sequence of varieties over $k$ with an action of $G\times\{S_n\}_n$ (with $G$ being a fixed linear algebraic group). We say that this sequence is \textit{motivically representation stable} if the sequences of varieties $[X_n/S_{\lambda[n]}]$ are motivically stable in $\widehat{M_\LL^G}$ for all partitions $\lambda$.
\end{definition}

In many of our cases, $G$ will be the trivial group. The following observation is key to understanding motivic representation stability.

\begin{example}\label{ex:freemotrepstab}
Let $X$ be a $G$-variety, and consider the sequence $X_n=X^n$ with the natural $G\times S_n$-action (where $S_n$ permutes the coordinates). Then, we see that
\[[X^n/S_{\lambda[n]}]=\Sym_G^{n-|\lambda|}X\times \prod_i \Sym_G^{\lambda_i}X\]
meaning that the sequence $X^n$ is motivically representation stable if and only if the sequence $\Sym^n_GX$ is motivically stable in $\widehat{M_\LL^G}$.
\end{example}

In the rest of the section, we motivate why we believe that Definition \ref{def:motrepstab} is a right definition for motivic representation stability in the context of the Grothendieck ring of varieties.

\subsection{Motivic decomposition with respect to representations}

Let $G$ be a finite group and consider, $R_\QQ(G)$, the rational representation ring of $G$ over $\ZZ$ that is the Grothendieck group of finite dimensional $\QQ$-representations of $G$ where the ring structure is induced by the tensor product of representations. We denote the trivial representation by $T_G$.

In this section, we wish to decompose a $G$-variety over a variety $S$ with respect to the rational representations of $G$, in other words, we wish to construct a ring map
\[\K_0(\Var_S^G)\to \K_0(\Var_S)\otimes_\ZZ R_\QQ(G)\]
so that for a variety $X$ over $S$ the part that corresponds to the trivial representation is $[X/G]$ (indeed, the tensor product on the right-hand side has a natural structure of a ring, because both $\K_0(\Var_S)$ and $R_\QQ(G)$ are rings). We denote the desired motivic decomposition by $[X]_G\in \K_0(\Var_S)\otimes_\ZZ R_\QQ(G)$ for a $G$-variety $X$. 

This goal will not be achieved in general. The main difficulties arise from a) the fact that the matrix of the inner product of rational representations over $\ZZ$ is not invertible, b) the lack of a motivic K\"unneth-formula, see Remark \ref{rem:counterkunneth}, and c) over-constraints coming from subgroups of $G$, see Example 3.6.9 of \cite{vogel2024motivicphd}.


The representation ring, $R_\QQ(G)$, is equipped with the standard inner product $\langle, \rangle$ that is given on two irreducible representations, $V_1$, $V_2$, by $\langle V_1, V_2\rangle=0$ if the irreducible representations are not isomorphic and $\langle V_1, V_2\rangle=1$, if the representations are isomorphic. This inner product can be extended to an inner product
\[\langle, \rangle: \K_0(\Var_S)\otimes_\ZZ R_\QQ(G)\otimes_\ZZ \K_0(\Var_S)\otimes_\ZZ R_\QQ(G)\to \K_0(\Var_S)\]
by $\langle [X_1]\otimes V_1, [X_2]\otimes V_2\rangle=[X_1][X_2]\otimes \langle V_1,V_2\rangle.$

\begin{definition}
    Let $G$ be a finite group, and let $\fH$ be a set of conjugacy
 classes of subgroups of $G$. We define the map
 \[\Psi^G_\fH: R_\QQ(G)\to \bigoplus_{[H]\in \fH}\ZZ\]
 as $[V]\mapsto \langle T_H, \Res_H^G(V)\rangle_{[H]\in \fH}$ sending a representation $V$ to the multiplicity of the trivial representation $T_H$ in $\Res_H^G(V)$. We say that a set of conjugacy classes of subgroups $\fH$ is \emph{very good} if the map $\Psi^G_\fH$ is an isomorphism of Abelian groups.
\end{definition}

\begin{remark}
    Not every finite group admits a very good set of conjugacy classes of subgroups.
\end{remark}

The remark above motivates the following definition.

\begin{definition}
We say that a set of conjugacy classes of subgroups $\fH$ is \emph{good} if the map
 \[\Psi^G_\fH\otimes \QQ: R_\QQ(G)\otimes_\ZZ \QQ\to \bigoplus_{[H]\in \fH}\QQ\]
defined as $[V]\otimes a\mapsto a\langle T_H, \Res_H^G(V)\rangle_{[H]\in \fH}$ is an isomorphism of vector spaces.
\end{definition}

\begin{remark}
    The reason we only consider conjugacy classes of subgroups is that both the multiplicity $\langle T_H, \Res_H^G(V)\rangle$ and the virtual class $[X/H]\in \K_0(\Var_S)$ only depend on the conjugacy class of the subgroup $H$.
\end{remark}

\begin{remark}
    We note that the inner product on $R_\QQ(G)$ naturally extends to an inner product 
\[\langle,\rangle: R_\QQ(G)\otimes_\ZZ \QQ\otimes_\QQ R_\QQ(G)\otimes_\ZZ \QQ\to \QQ\]
on the vector space $R_\QQ(G)\otimes_\ZZ \QQ$, and thus to an inner product on $\K_0(\Var_S)\otimes_{\ZZ}R_\QQ(G)\otimes_{\ZZ}\QQ$:
\[\langle,\rangle: (\K_0(\Var_S)\otimes_{\ZZ} R_\QQ(G)\otimes_\ZZ \QQ)\otimes_\QQ (\K_0(\Var_S)\otimes_\ZZ R_\QQ(G)\otimes_\ZZ \QQ)\to \K_0(\Var_S)\otimes_{\ZZ}\QQ\]
defined as 
\[\langle [X]\otimes V_1\otimes a, [Y]\otimes V_2\otimes b\rangle:=[X][Y]\otimes \langle V_1,V_2\rangle ab.\]
\end{remark}

We note that given a map of finite groups $H\to G$, we can extend the induction and restriction maps on the representation rings
\[\Ind_H^G: R_\QQ(H)\to R_\QQ(G) \mbox{ and } \Res_H^G: R_\QQ(G)\to R_\QQ(H)\]
to maps
\[\Ind_H^G:\K_0(\Var_S)\otimes_\ZZ\otimes R_\QQ(H)\to \K_0(\Var_S)\otimes_\ZZ\otimes R_\QQ(G)\]
\[[X]\otimes V\mapsto [X]\otimes \Ind_H^G(V)\]
and
\[\Res_H^G:\K_0(\Var_S)\otimes_\ZZ\otimes R_\QQ(G)\to \K_0(\Var_S)\otimes_\ZZ\otimes R_\QQ(H)\]
\[[X]\otimes V\mapsto [X]\otimes \Res_H^G(V).\]

\begin{remark}
    By extending the scalars, these maps can be extended to maps (which we denote the same way)
    \[\Ind_H^G:\K_0(\Var_S)\otimes_\ZZ\otimes R_\QQ(H)\otimes_\ZZ\QQ\to \K_0(\Var_S)\otimes_\ZZ\otimes R_\QQ(G)\otimes_\ZZ\QQ\]
    and
    \[\Res_H^G:\K_0(\Var_S)\otimes_\ZZ\otimes R_\QQ(G)\otimes_\ZZ\QQ\to \K_0(\Var_S)\otimes_\ZZ\otimes R_\QQ(H)\otimes_\ZZ\QQ\]
\end{remark}

Our key observation is that (very) good sets of conjugacy classes of subgroups provide motivic decompositions which is the main theorem of this section.

\begin{theorem}[Motivic Decomposition Theorem]\label{thm:motivicdecom}
    Let $\fH$ be a very good set of conjugacy classes of subgroups. Then, there exists a unique map of Abelian groups
    \[\K_0(\Var_S^G)\to \K_0(\Var_S)\otimes_\ZZ R_\QQ(G)\]
    so that for any $G$-variety $X$, the image that we denote by $[X]_G$ satisfies 
    \[\langle T_H, \Res_H^G[X]_G\rangle=[X/H]\]
    in $\K_0(\Var_S)$ for all $[H]\in \fH$.

    Similarly, if $\fH$ is a good set of conjugacy classes of subgroups, then, there exists a unique map of vector spaces
    \[\K_0(\Var_S^G)\otimes_\ZZ \QQ\to \K_0(\Var_S)\otimes_\ZZ R_\QQ(G)\otimes_\ZZ \QQ\]
    so that for any $G$-variety $X$, the image of $[X]_G$ satisfies 
    \[\langle T_H, \Res_H^G[X]_G\rangle=[X/H]\otimes 1\]
    in $\K_0(\Var_S)\otimes_\ZZ \QQ$ for all $[H]\in \fH$.
\end{theorem}

\begin{proof}
We only prove the second statement, the proof of the first statement is very similar.

Given a $G$-variety $X$, consider the potential element $[X]_G=\sum_{V} [U_V]\otimes [V]\otimes \lambda_V\in \K_0(\Var_S)\otimes_\ZZ R_\QQ(G)\otimes_\ZZ \QQ$ (where the summation goes over the isomorphism classes of irreducible rational representations $V$ of $G$). In order to satisfy our assumption, we need that
\[\sum_V [U_V]\otimes \lambda_V\langle T_H, V\rangle=[X/H]\otimes 1\]
holds for every $[H]\in \fH$. Since $\fH$ is a good set of conjugacy classes of subgroups, this equation system has a unique solution proving our statement.
\end{proof}

For many finite groups $G$, there does not exist a very good set of conjugacy classes of subgroups. Even though, a good set of conjugacy classes of subgroups exists for any finite group $G$, in many cases, there does not exist a natural choice.

\begin{example}
    Let $G$ be a finite cyclic group $\ZZ/n\ZZ$. In this case, consider the set of all subgroups $\fH$. It is easy to see that in this case, the number of subgroups and the number of rational representations agree (namely the number of divisors of $n$), and, thus, the corresponding map $\Psi_\fH\otimes \QQ$ is an isomorphism. However, other than the trivial case ($G=1$), there does not exist a very good set of conjugacy classes of subgroups.
\end{example}

\begin{example}
    Let $G_1$ and $G_2$ be finite groups and $\fH_1$ and $\fH_2$ be very good (and good respectively) sets of conjugacy classes of subgroups of $G_1$ and $G_2$ respectively. Then, $\fH_1\times \fH_2$ provides a very good (and good respectively) set of conjugacy classes of subgroups for $G_1\times G_2$. This implies that any finite Abelian group has a choice of a good set of conjugacy classes of subgroups, i.e for any finite Abelian group, a motivic decomposition exists. 
\end{example}

\begin{remark}
    In fact, a consequence of Artin's Induction Thoerem \cite[Corollary 1 in Chapter 13]{serre1977} tells us that for a finite group $G$ the number of isomorphism classes of irreducible rational representations agrees with the number of conjugacy classes of cylic subgroups of $G$. 
\end{remark}

\begin{example}\label{ex:goods_nsubgroups}
    Consider the symmetric group $S_n$, and for each partition $\lambda$ the subgroup $S_\lambda:=S_{\lambda_1}\times S_{\lambda_2}\times...$. Consider $\fH$ to be the set of conjugacy classes of the $S_\lambda$. Denote by $V_\lambda$ the irreducible representation corresponding to $\lambda$. Now, for $[V]=\sum a_\lambda [V_\lambda]\in R_\QQ(S_n)$, we have 
    \[\Psi_{S_n}^\fH(V)=\left(\langle T_{S_\lambda}, \Res_{S_\lambda}^{S_n} V\rangle\right)_{S_\lambda\in \fH}=\left(\langle \Ind_{S_\lambda}^{S_n}(T_{\S_\lambda}), V\rangle\right)_{S_\lambda\in \fH}=\left(\sum_\mu a_\mu K_{\mu,\lambda}\right)_{S_\lambda\in \fH}\]
    where $K_{\mu, \lambda}$ are the Kostka numbers (see \cite{fulton1991j} Corollary 4.39). Since, $K_{\mu, \mu}=1$ and $K_{\mu, \lambda}=0$ for $\mu<\lambda$ (with respect to the lexigraphic ordering), we have that $\Psi^{S_n}_\fH(V)$ is an isomorphism, and thus, $\fH$ provides a very good set of conjugacy classes of subgroups.
\end{example}

We show two explicit examples (the cases of $S_2$ and $S_3$) of the Motivic Decomposition Theorem following Example \ref{ex:goods_nsubgroups}.

\begin{example}\label{ex:mots2decomp}
    Consider $S_2$-varieties. In this case, Example \ref{ex:goods_nsubgroups} provides the subgroups $1$ and $S_2$. We denote the trivial representation of $S_2$ by $T$ and the sign representation by $\Sigma$. For an $S_2$-variety, the Motivic Decomposition Theorem (Theorem \ref{thm:motivicdecom}) provides a decomposition with respect to $T$ and $\Sigma$:
    \[[X]_{S_2}=[X]^+T+[X]^-\Sigma\]
    where $[X]^+:=[X/S_2]$ and $[X]^-:=[X]-[X]^+$.    
\end{example}

As an easy application, consider an elliptic curve $E$ with the negation action. In this case, $E/S_2=\PP^1$, and thus, $[E]^+=[\PP^1]$, and $[E]^-=[E]-[\PP^1]$. This shows that the Motivic Decomposition Theorem does not decompose the variety into subvarieties, but rather decomposes the virtual class into a linear combination of virtual classes of varieties.

\begin{example}\label{ex:mots3decomp}
    Consider $S_3$-varieties. In this case, Example \ref{ex:goods_nsubgroups} yields a set of three conjugacy classes of subgroups: the conjugacy classes of $1$, $S_2$, and $S_3$. We denote the trivial representation of $S_3$ by $T$, the sign representation by $\Sigma$ and the 2-dimensional irreducible representation by $R$. In this case, the Motivic Decomposition Theorem provides a decomposition:
    \[[X]_{S_3}=aT+b\Sigma+cR.\]
    We can compute $a$, $b$ and $c$ using the quotients $[X]$, $[X/S_2]$, and $[X/S_3]$. In fact, using the Motivic Decomposition Theorem tells that
    \[a+b+2c=[X], \quad a+c=[X/S_2], \quad a=[X/S_3].\]
    This means that $a$, $b$, and $c$ can be determined as
    \[a=[X/S_3],\quad b=[X]-2[X/S_2]+[X/S_3],\quad c=[X/S_2]-[X/S_3].\]
\end{example}

In the example above, we could have considered a different good set of conjugacy classes of subgroups.

\begin{example}
    Consider $S_3$-varieties, and the good set, $\fH'$, of conjugacy classes of subgroups given by $1$, $A_3$ and $S_3$. In fact, $\fH'$ is a good set of conjugacy classes of subgroups, since
    \[\langle T_{A_3}, \Res^{S_3}_{A_3}T\rangle=1, \langle T_{A_3}, \Res^{S_3}_{A_3}\Sigma\rangle=1, \langle T_{A_3}, \Res^{S_3}_{A_3}R\rangle=0;\]
    and
    \[\langle T_{1}, \Res^{S_3}_{1}T\rangle=1, \langle T_{1}, \Res^{S_3}_{1}\Sigma\rangle=1, \langle T_{1}, \Res^{S_3}_{1}R\rangle=2.\]
    Note that this computation also shows that $\fH'$ is NOT a very good set of conjugacy classes of subgroups. Using this good set of conjugacy classes of subgroups, we have
    \[[X]_{S_3}=aT+b\Sigma+cR\]
    where $a=[X/S_3]$, $b=[X/A_3]-[X/S_3]$, and $c=\frac{1}{2}([X]-[X/A_3]+[X/S_3])$. The classes $[X]$, $[X/S_2]$, $[X/A_3]$ and $[X/S_3]$ are in general independent in $\K_0(\Var_S)$ (see \cite{Sawin}). This shows that using $\fH'$ we obtain a different motivic decomposition than in Example \ref{ex:mots3decomp}. This illustrates that the motivic decomposition in Theorem \ref{thm:motivicdecom} indeed depends on the choice of a good set of conjugacy classes of subgroups.
\end{example}

Using the specific choice of a good set of conjugacy classes of subgroups of $S_n$, we have the following statement that motivates Definition \ref{def:motrepstab}.

\begin{theorem}[Proposition 7.3.3 of \cite{vogel2024motivicphd}]\label{thm:motrepstabequivalent}
    Let $\{X_n\}_n$ be a sequence of varieties over $k$ equipped with an action of the symmetric groups $\{S_n\}_n$. Then, $\{X_n\}_n$ is motivically representation stable, if and only if, for any partition $\lambda$, in the motivic decomposition, the coefficient of $[X_n]_{S_n}$ of $V_{\lambda[n]}$ is motivically stable. In other words, the motivic decomposition is motivically stable.\qed
\end{theorem}

\begin{proof}
    For the sake of completeness, we include the proof.
    
    Assume first that the sequence of varieties $\{X_n\}_n$ is motivically representation stable. Consider the sequence of partitions $\lambda[n]$. The partitions $\mu$ that satisfy $\mu>\lambda[n]$ with respect to the lexicographic ordering are of the form $\mu=\lambda'[n]$ for some partition $\lambda'>\lambda$. Therefore, the coefficient of $\lambda[n]$ of $[X_n]_{S_n}$ in the motivic decomposition is a $\ZZ$-linear combination of the classes $[X_n//S_{\lambda'[n]}]$ implying one direction.

    The proof of the other direction is similar. Assume now, that the coefficients of $[X_n]_G$ are motivically stable. We proceed by induction. First, $[X_n//S_n]$ is motivically stable, because that is the coefficient of the trivial representation. Now, assume that we showed that $[X_n//S_{\lambda'[n]}]$ are motivically stable for all $\lambda'>\lambda$. The coefficient of $\lambda[n]$ in the motivic decomposition is a $\ZZ$-linear combination of $[X_n//S_{\mu[n]}]$ with $\mu\geq \lambda$, moreover the coefficient of $[X_n//S_{\lambda[n]}]$ is 1 in this linear combination. This means that $[X_n//S_{\lambda[n]}]$ is a $\ZZ$-linear combination of the coefficients of $[X_n]_G$ for representations $\mu[n]$ with $\mu\geq \lambda$ implying the other direction via the induction.
\end{proof}

\subsubsection{K\"unneth-formula}

The vector space $\K_0(\Var_S^G)\otimes_\ZZ R_\QQ(G)\otimes_\ZZ \QQ$ has a natural ring structure coming from the ring structures of $\K_0(\Var_S)$ and $R_\QQ(G)$ and $\QQ$. In this section, we investigate K\"unneth-tpye formulae for motivic decompositions, namely whether the multiplication structure on $\K_0(\Var_S^G)\otimes_\ZZ R_\QQ(G)\otimes_\ZZ \QQ$ is compatible with the motivic decomposition.

We show that although the K\"unneth-formula does not hold in general (see Remark \ref{rem:counterkunneth}) for the motivic decomposition, it holds for simple cases (Theorem \ref{thm:motkunneth}). To show these results, we start with an easy lemma.

\begin{lemma}[Lemma 3.6.18. of \cite{vogel2024motivicphd}]
    Let $G$ be a finite algebraic group, $\fH$ a good set of conjugacy classes of subgroups of $G$. Let $\fV_G^\fH\subseteq \K_0(\Var^G_S)$ be the subset of elements $[X]$ so that $[XY]_G=[X]_G[Y]_G$. Then, $\fV_G^\fH$ is a $\K_0(\Var_S)$-subalgebra of $\K_0(\Var_S^G)$.\qed
\end{lemma}

\begin{proof}
    Since $[Y]_G$ is $\K_0(\Var_S)$-linear, therefore $\fV_G^\fH$ is a $\K_0(\Var_S)$-linear subgroup of $\K_0(\Var_S^G)$. To see that $\fV_G^\fH$ is closed under multiplication, consider $[X], [Y]\in \fV_G^\fH$ and $[Z]\in \K_0(\Var_S^G)$, and take
    \[[(XY)Z]_G=[X(YZ)]_G=[X]_G[YZ]_G=[X]_G[Y]_G[Z]_G=[XY]_G[Z]_G.\]
\end{proof}

This lemma enables us to provide a large set of varieties for which the motivic K\"unneth-formula holds.

\begin{theorem}\label{thm:motkunneth}[Theorem 3.6.19. of \cite{vogel2024motivicphd}]
    Let $G$ be a finite algebraic group over a field $k$ so that $k$ contains all $|G|$-th roots of unity, $\fH$ a good set of conjugacy classes of subgroups of $G$. Then,
    \begin{itemize}
        \item If $G$ acts linearly on $\AA^1$, then $\fV_G^\fH$ contains the class of $\AA^1$.
        \item If $G$ acts diagonally on $\AA^n$, then $\fV_G^\fH$ contains the class of $\AA^n$.
        \item If $G$ acts diagonally on $\PP^n$, then $\fV_G^\fH$ contains the class of $\PP^n$.
        \item The discreet group scheme $G$ with the natural translation action by $G$ lies in $\fV_G^\fH$. 
    \end{itemize}
\end{theorem}

\begin{proof}
    We provide an alternative short proof for the first statement, then the second and third statement follows from the first one from the lemma before.
    
    Since, $[\AA^1/H]=[\AA^1]$ for any linear action of a group on $[\AA^1]$, we have that the motivic decomposition of $[\AA^1]$ is just $[\AA^1]\otimes T_G$. Thus, to show the first statement, we need to show that $[\AA^1\times Y/H]=[\AA^1][Y/H]$ for all $Y\in \K_0(\Var_S^G)$ and all $[H]\in fH$.  

    Consider, $N\leq H$, the kernel of the representation $H\to \GL^1(k)$ corresponding to the linear action of $H$ on $\AA^1$. Then, 
    \[[\AA^1\times Y/H]=[(\AA^1\times Y/N)/(H/N)]=[\AA^1][(Y/N)/(H/N)]=[\AA^1][Y/H]\]
    showing that if the statement is true for the smaller group $H/N$, then it is also true for $H$. Therefore, using induction, we can assume that $H$ is a cyclic group acing freely on $\GG_m=\AA^1\setminus \{0\}$. In this case, $[\GG_m\times Y/H]$ is a $\GG_m$-torsor over $[Y/H]$, so
    \[[\GG_m\times Y/H]=[\GG_m][Y/H]\]
    impyling the proof as $[\AA^1\times Y/H]=[\GG_m\times Y/H]+[Y/H]$.
\end{proof}

We conclude this section by explicitly describing the Motivic K\"unneth-formula in the cases of $S_2$ and $S_3$.

\begin{example}
    Consider $S_2$ and the motivic decomposition as in Example \ref{ex:mots2decomp}, i.e for a class $[X]\in \K_0(\Var^{S_2}_S)$, we have a decomposition
    \[[X]_{S_2}=[X]^+T+[X]^-\Sigma.\]
    Now, let $[Y]\in \fV^{\fH}_{S_2}$. Then,
    \[[X\times Y]_{S_2}=[X]_{S_2}[Y]_{S_2}.\]
    In other words, 
    \[[X\times Y]^+=[X]^+[Y]^++[X]^-[Y]^-\]
    and
    \[[X\times Y]^-=[X]^+[Y]^-+[X]^-[Y]^+.\]
\end{example}

\begin{remark}\label{rem:counterkunneth}
    This example highlights why the Motivic K\"unneth-formula cannot hold in general. Indeed, consider an elliptic curve $X=Y=E$ with the negation action. Then, the surface $E\times E/S_2$ is a singular surface, whose resolution is the Kummer K3 surface. By \cite{larsen2003motivic}, there exists a motivic measure 
    \[\K_0(\Var_k)\to \ZZ[SB_k]\]
    from the Grothendieck ring of varieties to the free Abelian group of stable birational classes of smooth and projective varieties that induces an isomorphism $\K_0(\Var_k)/[\AA^1]\cong \ZZ[SB_k]$. Under this isomorphism, the class of a K3 surface is not a linear combination of the classes of $[\PP^1\times \PP^1]$, $[\PP^1\times E]$, $[E\times \PP^1]$ and $[E\times E]$, and thus, the virtual class $[E\times E/S_2]\in \K_0(\Var_k)$ is also linearly independent from the virtual classes of $[\PP^1\times \PP^1]$, $[\PP^1\times E]$, $[E\times \PP^1]$ and $[E\times E]$.
\end{remark}

We note that a similar counterexample appears in the PhD thesis of the second author, namely Example 3.6.17. of \cite{vogel2024motivicphd}.

\begin{example}
    Consider $S_3$ and the motivic decomposition as in Example \ref{ex:mots3decomp}, i.e, we have that for a class $[X]\in \K_0(\Var^{S_3}_S)$ a decomposition
    \[[X]_{S_3}=aT+b\Sigma+cR.\]
    Given $[Y]\in \fV^{\fH}_{S_3}$ with $[Y]=a'T+b'\Sigma+c'R$, we have
    \[[X\times Y]_{S_3}=[X]_{S_3}[Y]_{S_3}.\]
    In other words,
    \[[X\times Y]_{S_3}=(aa'+bb'+cc')T+(ab'+a'b+cc')\Sigma+(ac'+bc'+a'c+b'c+cc')R.\]
\end{example}

\subsection{Compatibility with mixed Hodge structures}
In this section, we show that Theorem \ref{thm:motivicdecom} is compatible with mixed Hodge structures. We follow \cite{cappell2008hodge} for basic definitions.

\begin{definition}
    We define the Grothendieck ring of mixed Hodge structures, $\K_0(\MHS)$, as the Grothendieck group of the Abelian category of (graded polarizable) rational mixed Hodge structures. The ring structure is given by tensor product, and the unit is given by the 1-dimensional pure Hodge structure of weight 0.
\end{definition}

The $E$-polynomial motivic measure
\[\K_0(\Var_\CC)\to \ZZ[u,v]\]
factors through the Grothendieck ring of mixed Hodge structures
\[\K_0(\Var_\CC)\xrightarrow{D} \K_0(\MHS)\to \ZZ[u,v]\]
where the first map is defined on a variety $X$
\[D:[X]\mapsto \sum_i (-1)^i [H^i_c(X, \QQ)]\]
as the alternating sum of the mixed Hodge structures (defined by Deligne \cite{PMIHES_1971__40__5_0, PMIHES_1974__44__5_0}) on the compactly supported cohomology of $X$; and the second map is defined as the E-polynomial of a mixed Hodge structure
\[(V, F^\bullet ,W_{\bullet})\mapsto \sum_{p,q} \dim_{\CC} gr^p_Fgr_{p+q}^W (V \otimes \CC)u^pv^q.\]

Given a variety $X$ with an action of a finite group $G$, the mixed Hodge structure of $X$ is also equipped with a $G$-action that can be extended to a map 
\[D_G:\K_0(\Var_{\CC}^G)\to \K_0(\MHS^G)\]
from the Grothendieck ring of varieties with a $G$ action to the Grothendieck ring of the category of mixed Hodge structures with a $G$-action, $\MHS^G$. We note that $\MHS^G$ is an Abelian category, and thus, any object $H\in \MHS^G$ admits a canonical decomposition
\[H=\bigoplus_{\rho\in Irr_{\QQ}(G)}H_\rho\otimes \rho\]
with respect to the irreducible rational representation of $G$.  Therefore, we obtain a homomorphism of Abelian groups
\[\Phi:\K_0(\MHS^G)\to \K_0(\MHS)\otimes_{\ZZ}R_\QQ(G),\]

The discussion above yields the following theorem.

\begin{theorem}\label{thm:VarvsMHS}
    Let $\fH$ be a very good set of conjugacy classes of subgroups of a finite group $G$. The motivic decomposition map $-_{G}:\K_0(\Var_\CC^G)\to \K_0(\Var_\CC)\otimes_{\ZZ} R_\QQ(G)$ corresponding to $\fH$ fits into a commutative diagram
    \[
    \begin{tikzcd}
        \K_0(\Var_\CC^G)\ar[r, "-_G"]\ar[d, "D_G"] & \K_0(\Var_\CC)\otimes_{\ZZ} R_\QQ(G)\ar[d, "D\otimes \id"]\\
        \K_0(\MHS^G)\ar[r, "\Phi"]& \K_0(\MHS)\otimes_{\ZZ}R_\QQ(G).
    \end{tikzcd} \]
\end{theorem}

\begin{proof}
    We need to show that $(D\otimes \id)\circ -_G=\Phi\circ D_G$. Since $\fH$ is a very good set of conjugacy classes of subgroups, it is enough to show that
     \[\langle T_H, \Res_H^G((D\otimes \id)[X]_G)\rangle=\langle T_H, \Phi(D_G(X))\rangle \]
    holds for any $H\in \fH$ and any variety $X$ with a $G$-action. The left-hand side can be identified with
    \[\langle T_H, \Res_H^G((D\otimes \id)[X]_G)\rangle=D([X/H])\]
    while the right-hand side can be identified with the $H$-invariant part of $D_G(X)$. These two agree as the $H$-equivariant part of the mixed Hodge structure is given by the mixed Hodge structure on $X/H$ (see \cite{florentino2025mixed}):
\[H^{k,p,q}(X/H, \CC)=H^{k,p,q}(X)^H\]
proving our statement.
\end{proof}

We now connect motivic representation stability with representation stability of the $E$-polynomial. We start by defining the stability of $E$-polynomials.

\begin{definition}\label{def:Epolystable}
    Let $\{X_n\}_n$ be a sequence of varieties. We say that the $E$-polynomial of $\{X_n\}_n$ is stable if the limit
    \[\lim_{n\to \infty}\frac{E(X_n)}{(uv)^{\dim X_n}}\]
    exists in the completed localized ring $\widehat{\ZZ[u,v,\frac{1}{uv}]}$ (defined in Remark \ref{rem:epolymotmeasure}).
\end{definition}

\begin{remark}
    Roughly, in simple terms, the stability of the E-polynomial tells us that the top parts of the Hodge diamonds of the sequence of varieties $\{X_n\}_n$ converge. 
\end{remark}

Now, we define the representation stability of $E$-polynomials.

\begin{definition}
    Let $\{X_n\}_n$ be a sequence of varieties with an action of $\{S_n\}_n$. We say that the $E$-polynomial of $\{X_n\}_n$ is representation stable if the $E$-polynomials of the sequences of varieties $[X_n/S_{\lambda[n]}]$ are stable for all partitions $\lambda$.
\end{definition}

\begin{remark}
    An easy adjustment of Theorem \ref{thm:motrepstabequivalent} tells us that given a sequence of varieties $\{X_n\}_n$ with an action of the symmetric groups $\{S_n\}_n$, we have that the $E$-polynomial of $\{X_n\}_n$ is representation stable, if and only if, for any partition $\lambda$ the $E$-polynomial of the coefficient of $[X_n]_{S_n}$ of $V_{\lambda[n]}$ is stable.
\end{remark}

As an easy consequence of Theorem \ref{thm:VarvsMHS} we get the following.

\begin{corollary}
    Let $\{X_n\}_n$ be a sequence of varieties with an action of the symmetric groups $S_n$ so that $\{X_n\}_n$ is motivically representation stable. Then, the $E$-polynomial of $\{X_n\}_n$ is representation stable.\qed
\end{corollary}

Again, in simple terms, for any partition $\lambda$, the top parts of the Hodge diamonds of $\{X_n/S_{\lambda[n]}\}$ converge. 

\section{Motivic Stability of Representation Varieties of Surface groups}\label{sec:surface}

Let $M_g$ be a compact genus $g$ smooth surface, and $G$ a linear algebraic group. In this section, we are interested whether the varieties $\Rep_G(M_g):=\Hom_{Gp}(\pi_1(M_g), G)$ are motivically stable in terms of $g$. A priori, $\Hom_{Gp}(\pi_1(M_g), G)$ is only endowed with a set structure, however, 
\[\Hom_{Gp}(\pi_1(M_g), G)=\{A_1,B_1,...,A_g,B_g\in G^{2g}|\prod_{i=1}^g [A_i,B_i]=\id\}\]
and thus $\Rep_G(M_g)$ can, indeed, be realized as a closed subvariety of $G^{2g}$.

Explicitly, we are interested whether the limit
\[\lim_{g\to \infty}\frac{[\Rep_G(M_g)]}{[G^{2g}]}\]
exists in the completed Grothendieck ring of stacks, $\widehat{\K_0(\Stck_k)}$. Informally, this limit measures the probability that the $2g$ elements $A_1, B_1, ..., A_g, B_g$ of the group $G$ satisfy the property
\[\prod_{i=1}^g [A_i,B_i]=\id.\]

\subsection{Finite field heuristics}

We expect that as the genus gets larger and larger the product of commutators $\prod_{i=1}^g [A_i,B_i]$ spreads evenly across the commutator subgroup $[G,G]$. In fact, this happens in the case of finite groups $G$. In this case, using a classical result of Frobenius \cite{frobenius1896gruppencharaktere}, $\Rep_G(M_g)$ is a finite set of cardinality
\begin{equation}\label{eq:frobenius}
    \#\Rep_G(M_g)=\#G\sum_{\chi\in \Irr(G)} \left(\frac{\#G}{\chi(1)}\right)^{2g-2}
\end{equation}
that depend only on the dimensions $\chi(1)$ of irreducible (complex) characters of $G$. Thus, in the case of finite groups, Limit \ref{eq:limit} exists, namely
\[\lim_{g\to \infty}\frac{\#\Rep_G(M_g)}{\#G^{2g}}=\lim_{g\to \infty}\frac{1}{\#G}\sum_{\chi}\frac{1}{\chi(1)^{2g-2}}=\frac{1}{[G,G]}.\]

The above computation suggests that such limit should exist motivically in the case of connected linear algebraic groups over $\CC$.

\subsection{The case of reductive groups with connected centers}

In this section, we show that Limit \ref{eq:limit} exists in the context of the motivic measure given by the E-polynomial (see Remark \ref{rem:epolymotmeasure}) for reductive groups $G$ with a connected center. In other words, we show that E-polynomial of $\Rep_G(M_g)$ is stable (see Definition \ref{def:Epolystable}).

\begin{proposition}\label{prop:limitreductive}
    Let $G$ be a reductive group over $\CC$ with a connected center. Then, 
    \[\lim_{g\to \infty} \frac{E(\Rep_G(M_g))}{E(G)^{2g}}=\frac{E(G/[G,G])}{E(G)}.\]
\end{proposition}

\begin{proof}
    The $G$-representation varieties of surfaces, $\Rep_G(M_g)$, are PORC count (polynomial on residue classes) \cite{BRIDGER_KAMGARPOUR_2023} meaning that $\# \Rep_{G(\FF_q)}(M_g)$ is a polynomial, $q_{G,g}(t)\in \QQ[t]$, in terms of $q$ assuming that $q\equiv 1$ module $d(G^\vee)$, the modulus of the Langlands dual group. Using a result of Katz (see Appendix of \cite{hausel2008mixed}), this implies that the E-polynomial of the representation varieties, $\Rep_G(M_g)$, is a polynomial of $uv$, explicitly, the E-polynomial is given as $q_{G,g}(uv)$. 

We analyze the point count, $q_(G,g)(t)$, now in terms of the dimensions of the representations. The summand of the right-hand side of Equation \ref{eq:frobenius} corresponding to 1-dimensional representations is given as
\[\#G\cdot \#(G/[G,G])\cdot (\#G)^{2g-2}.\]
This is also PORC count (we denote it by $h_{G,g}(t)$) providing a contribution of $E(G/[G,G])\cdot E(G)^{2g-1}$ in the E-polynomial, $E(\Rep_G(M_g))$. The rest of the summands is also PORC count since the rest of the summands corresponds to the polynomial $q_{G,g}(t)-h_{G,g}(t)$. Using Remark 5.2 of \cite{BRIDGER_KAMGARPOUR_2023}, the rest of the summands (that are of the form $\left(\frac{G}{\chi(1)}\right)$) are of smaller dimensions than $G$. Therefore, their contribution is negligible in the limit (since we raise these pieces to the $2g$-th power). Thus, this holds for the E-polynomial as well implying our statement.
\end{proof}

\subsection{General linear algebraic groups}
We conjecture, based on the finite field heuristics and Proposition \ref{prop:limitreductive}, that the limit \ref{eq:limit} exists.

\begin{conjecture}\label{con:stabsurface}
    Let $G$ be a connected linear algebraic group over $k$. Let $\Rep_G(M_g)$ denote the $G$-representation variety of the surface group of a smooth compact genus $g$ surface. Then, 
    \[\lim_{g\to \infty}\frac{[\Rep_G(M_g)]}{[G^{2g}]}=\frac{[G/[G,G]]}{[G]}\]
    in the completed Grothendieck rings of stacks, $\widehat{\K_0(\Stck_k)}$.
\end{conjecture}


We verify the conjecture in two main cases: 1. in the case of $G=\SL_2(k)$ and in the case of upper triangular matrix groups $G$ of rank up to 5. In fact, the virtual classes of $G$-representation varieties of surface groups are notoriously difficult to compute, namely, only the cases above have been solved successfully \cite{thesisangel, habvog20, vogel2024motivic}. In all such cases, the computations were performed using Topological Quantum Field Theories. 

\begin{example}[The case of upper triangular matrices]
 Explicit computations have been done for the group of upper triangular matrices or unipotent matrices of ranks 2, 3, 4 and 5 \cite{habvog20, vogel2024motivic}:
\begin{itemize}
    \item $[\Rep_{\U_2}(M_g)]=q^{2g-1}(q-1)^{2g+1}\left((q-1)^{2g-1}+1\right)=q^{2g-1}(q-1)^{4g}+l.o.t$
    \item $[\Rep_{\U_3}(M_g)]=q^{6g-3}(q-1)^{6g}+l.o.t$
    \item $[\Rep_{\U_4}(M_g)]=q^{12g-6}(q-1)^{8g}+l.o.t$
    \item $[\Rep_{\U_5}(M_g)]=q^{20g-10}(q-1)^{10g}+l.o.t$
\end{itemize}
providing the proof of Conjecture \ref{con:stabsurface} for the low rank upper triangular matrix groups, indeed, we have
\begin{itemize}
    \item $\lim_{g\to \infty}\frac{[\Rep_{\U_2}(M_g)]}{[\U_2]^{2g}}=\lim_{g\to \infty} \frac{q^{2g-1}(q-1)^{4g}+l.o.t}{q^{2g}(q-1)^{4g}}=\frac{1}{q}=\frac{1}{[[\U_2, \U_2]]},$
    \item $\lim_{g\to \infty}\frac{[\Rep_{\U_3}(M_g)]}{[\U_3]^{2g}}=\lim_{g\to \infty} \frac{q^{6g-3}(q-1)^{6g}+l.o.t}{q^{6g}(q-1)^{6g}}=\frac{1}{q^{3}}=\frac{1}{[[\U_3, \U_3]]},$
    \item $\lim_{g\to \infty}\frac{[\Rep_{\U_4}(M_g)]}{[\U_4]^{2g}}=\lim_{g\to \infty} \frac{q^{12g-6}(q-1)^{8g}+l.o.t}{q^{12g}(q-1)^{8g}}=\frac{1}{q^{6}}=\frac{1}{[[\U_4, \U_4]]},$
    \item $\lim_{g\to \infty}\frac{[\Rep_{\U_5}(M_g)]}{[\U_5]^{2g}}=\lim_{g\to \infty} \frac{q^{20g-10}(q-1)^{10g}+l.o.t}{q^{20g}(q-1)^{10g}}=\frac{1}{q^{10}}=\frac{1}{[[\U_5, \U_5]]}.$
\end{itemize}
\end{example}

\begin{example}[The case of $\SL_2(k)$]
     Using the explicit description of eigenvectors and eigenvalues \cite{vogel2020topological} one gets
 \[ [\Rep_{\SL_2(k)}(M_g)]= \frac{1}{2}q^{2g-1}(q + 1)^{2g-1}(q-1)(2^{2g} + q - 1) + \frac{1}{2}q^{2g-1}(q -1)^{2g-1}(q+1)(2^{2g} + q-3) +\]
 \[+(q^{2g-1} +q)(q-1)^{2g-1}(q+1)^{2g-1}\]
 showing  \[\lim_{g\rightarrow \infty}\frac{[\Rep_{\SL_2(k)}(M_g)]}{[\SL_2(k)]^{2g}}=\frac{1}{q(q-1)(q+1)}=\frac{1}{[\SL_2(k),\SL_2(k)]}\]
     in $\widehat{\K_0(\Stck_k)}$.
\end{example}

\subsection{Counterexamples via non-connected groups}

The assumption that the linear algebraic group is connected is crucial in Conjecture \ref{con:stabsurface}. In fact, in the case of non-connected groups, Lang's theorem \cite{lang1956algebraic} fails, and the finite field heuristics fail. To give an example, we consider $G = \GG_m \rtimes \ZZ/2\ZZ$, where the action of $\ZZ/2\ZZ$ on $\GG_m$ is given by $x \mapsto x^{-1}$. In \cite{gonzalez2022virtual}, the authors described the virtual class of the corresponding representation variety as
\[[\Rep_G(M_g)] = (q - 1)^{2g - 1} (q - 3 + 2^{2g + 1}) \]
providing
\[\lim_{g\rightarrow \infty} \frac{[\Rep_G(M_g)]}{[G]^{2g}}=\frac{2}{q-1}\ne \frac{1}{[G,G]}=\frac{1}{q-1}\]
in $\K(\Stck/k)$. 

The problem is even more severe if one considers the corresponding character stack. In fact, in \cite{gonzalez2022virtual}, the class of the character stack is computed as
\[ [\mathfrak{X}_G(M_g)]) = \frac{\left(q - 1\right)^{2 g - 2} \left(2^{2 g + 1} + q - 3\right)}{2} + \frac{\left(q + 1\right)^{2 g - 2} \left(2^{2 g + 1} + q - 1\right)}{2}.\]
Using this computation, it is clear that the limit in Conjecture \ref{con:stabsurface} does not exist.
\section{Motivic representation stability of representation varieties of free groups and free Abelian groups} Goal of the section is Motivic representation stability of representation varieties of free groups and free Abelian groups

\subsection{Free groups}

In this section, we study the motivic representation stability of the representation varieties $G^n:=\Hom_{Gp}(F^n, G)$ corresponding to the free groups $F_n$ on $n$ letters, and of the corresponding character stacks $[G^n/G]$ where the group $G$ acts by conjugation on $G^n$. The symmetric groups, $S_n$, act by permuting the coordinates of $G^n$, and the conjugation action by the group $G$ commutes with the action of $S_n$.  

First, we show that the representation varieties are motivically representation stable which follows from the following easy lemma.

\begin{lemma}\label{lem:motstabformonic}
    Let $X$ be a variety over $k$, so that $[X]$ is a polynomial in $\LL$ in $\K_0(\Var_k)$. Then, the symmetric powers, $\Sym^n(X)$, are motivically stable if and only if $[X]$ is a monic polynomial of $\LL$.
\end{lemma}

\begin{proof}
   We compute the motivic zeta function of $X$ from the polynomial form $[X]=\sum_{i=0}^s a_i\LL^i$ using that the motivic zeta function is motivic:
   \[Z(X,t)=\prod_{i=0}^s\left(\sum_{n=0}^\infty [\Sym^n(\LL^i)]t^n\right)^{a_i}.\]
   Using that $[\Sym^n(\LL^{i})]=\LL^{ni}$ (see for instance, Lemma 4.4 of \cite{gottsche2000motive}), we have
   \[Z(X,t)=\prod_{i=0}^s\left(\sum_{n=0}^\infty \LL^{ni}t^n\right)^{a_i}=\prod_{i=0}^s \frac{1}{(1-q^it)^{a_i}}.\]
   Thus, if $a_s=1$, then
   \[\lim_{n\to \infty}\frac{[\Sym^n(X)]}{\LL^{ns}}=\left((1-t)Z(X,t/\LL^s)\right)|_{t=1}=\frac{1-t}{\prod_{i=0}^s (1-\LL^{i-s}t)^{a_i}}=\frac{1}{\prod_{i=0}^{s-1} (1-\LL^{i-s}t)^{a_i}}\]
   and if $a_s>1$, it is easy to see that the limit does not exist.
\end{proof}

This lemma above shows that the $G$-representation varieties corresponding to free groups where $G$ is a connected linear algebraic group are motivically representation stable proving Conjecture \ref{introcon:B} using Example \ref{ex:freemotrepstab}. Indeed, $G^n$ is motivically representation stable if and only if $\Sym^n G$ is motivically stable as a sequence of $G$-varieties.

Now, we turn our attention to the character stacks of free groups. Again, by Example \ref{ex:freemotrepstab}, $G^n$ is motivically representation stable if and only if $\Sym^n_G G$ is motivically stable as a sequence of $G$-varieties.

The key examples we consider (similarly to the other examples in the paper) are $G=\GL_r$ and the groups of upper triangular matrices. The key statement we need is the following.

\begin{proposition}\label{prop:symmstabilizeGLr}
Consider the natural action of $\GL_r$ on $\AA^r$. Then,
\[\lim_{n\to \infty}\frac{[\Sym^n_{\GL_r}(\AA^r)]}{\LL^{nr}}=\frac{[\GL_r]}{\prod_{i=1}^r (\LL^r-\LL^i)}\]
in $\widehat{M_\LL^G}$ where $\GL_r$ acts free on $[\GL_r]$ by left multiplication.
\end{proposition}

\begin{proof}
Let us denote by $X_n$ the part of $\Sym^n(\AA^r)$ where $\GL_r$ acts freely. We claim that
\[\lim_{n\to \infty}\frac{[X_n]}{\LL^{nr}}=\lim_{n\to \infty}\frac{[\Sym^n(\AA^r)]}{\LL^{nr}}=1\]
holds in $\widehat{M_\LL}$. 

Recall that an algebraic group $G$ is called \textit{special} if any $G$-torsor is Zariski-locally trivial, in other words, if $E\to X$ is a $G$-torsor, then the relation $[E]=[G][X]$ holds in the Grothendieck-ring of varieties. The group $\GL_r$ is a special group, and thus, our claim implies the statement, because in this case
\[\lim_{n\to \infty}\frac{[X_n]}{\LL^{nr}}=\lim_{n\to \infty}\frac{[X_n/\GL_r][\GL_r]}{\LL^{nr}}=\frac{[\GL_r]}{\prod_{i=1}^r (\LL^r-\LL^i)}\]
holds in $\widehat{M_\LL^G}$ since the $\GL_r$-action on $[X_n/\GL_r]$ is trivial, so $[X_n/\GL_r]$ can be replaced by $\frac{[\Sym^n(\AA^r)]}{\GL_r}$. In order to show the claim we will show that $X_n$ can be covered by varieties of negligible dimension. 

In fact, we have a cover of the form
\[\bigsqcup_{\sigma\in S_n} \{A, (x_1,...,x_n)\in (\GL_r\setminus id)\times \AA^{rn}| \forall i:Ax_i=x_{\sigma(i)} \}.\]
This cover is of dimension at most $\dim \GL_r + nr-n$, because either 1) all $x_i$'s have to be eigenvectors of $A$ with eigenvalue 1 or 2) at least one of the equations $x_i=Ax_j$ hold with $i\ne j$, so $x_i$ is determined by $x_j$. This concludes the proof of the statement.
\end{proof}

This proposition implies an important technical statement that was used in \cite{ramras2021homological} in the context of FI-modules.

\begin{corollary}\label{cor:stabsymmetric}
    Let $G$ be a linear algebraic group acting on $\AA^r$ via a homomorphism $\phi:G\to \GL_r$. Then,
    \[\lim_{n\to \infty}\frac{[\Sym^n_G(\AA^r)]}{\LL^{nr}}=\frac{[\GL_r]}{\prod_{i=1}^r (\LL^q-\LL^i)}\]
    holds in $\widehat{M_\LL^G}$ where $G$ acts on $\GL_r$ via the homomorphism $\phi$.
\end{corollary}

\begin{proof}
    We have that $\Res_H^G$ provides a continuous map $\widehat{M_\LL^G}\to \widehat{M_\LL^H}$ using Corollary \ref{cor:resindmotivic}. Note that
    \[\Res^{\GL_r}_G \circ \Sym^n_{\GL_r}=\Sym^n_G\circ \Res^{\GL_r}_G\]
    which implies our statement.
\end{proof}

Using the corollary above, we are ready to show that the $G$-character stacks corresponding to free groups satisfy motivic representation stability in the cases of $G=\GL_r$ or $G=\U_r$.

\begin{theorem}\label{thm:conjB}
    The $G$-character stacks corresponding to free groups satisfy motivic representation stability in the cases of $G=\GL_r$ or $G=\U_r$.
\end{theorem}

\begin{proof}
    In order to show motivic representation stability, it is enough to show that the sequence of symmetric powers $\Sym^n_G(X)$ is motivically stable (see Example \ref{ex:freemotrepstab}). In the case of $G=\GL_r$, we consider the linear action of $\GL_r$ on the variety of $[r\times r]$ matrices (which is isomorphic to $\AA^{r^2}$). By Corollary \ref{cor:stabsymmetric}, we get that $\Sym^n_{\GL_r}(\AA^{r^2})$ is motivically stable. Since $\GL_r$ is an open subset of $\AA^{r^2}$, Lemma  \ref{lem:openstabilizeallstabilize} implies our statement. The case of $G=\U_r$ is very similar, we consider the embedding $\U_r\to \GL_r$, and the linear action of the group of invertible upper triangular matrices on variety of upper triangular matrices (which is an affine space), and thus Corollary \ref{cor:stabsymmetric} together with Lemma \ref{lem:openstabilizeallstabilize} imply the statement.
\end{proof}

\subsection{Free Abelian groups}

In this section, we study the motivic representation stability of the representation varieties $C_n(G):=\Hom_{Gp}(\ZZ^n, G)$ corresponding to the free Abelian groups $\ZZ^n$, and of the corresponding character stacks $[C_n(G)/G]$. As before, the symmetric groups, $S_n$, act by permuting the coordinates of $G^n$, and the group $G$ acts by conjugation. 

\subsubsection{Branching matrix}

First, we provide a computational technique to find the virtual classes of the variety of commuting $n$-tuples, $C_n(G)$, for reductive groups via branching matrices \cite{sharma20}. We illustrate the method for the reductive group $\GL_2(k)$ and $\GL_3(k)$.


\textbf{The variety of commuting $n$-tuples of $\GL_2(k)$}

We begin by understanding the virtual classes of subvarieties of $\GL_2(k)$. 

The virtual class of $\GL_2(k)$ is $q(q-1)(q^2-1)$. It has three types of elements
\begin{itemize}
    \item scalar matrices: the closed subvariety of scalar matrices has virtual class $q-1$.
    \item $J$-type: the open subvariety of the closed subvariety of matrices with the same eigenvalue (but not the scalar matrices) has virtual class $(q-1)(q^2-1)$. These are the matrices that are conjugate to a matrix of the form
    \[\left(\begin{array}{cc}\lambda & 1\\0 & \lambda\end{array}\right).\]
    \item $M$-type: the open subvariety of matrices with two distinct eigenvalues has virtual class $(q-1)(q^3-q^2-q)$. These are the matrices that are conjugate to a matrix of the form 
    \[\left(\begin{array}{cc}\lambda & 0\\0 & \mu\end{array}\right)\]
    with $\lambda\ne \mu$.
\end{itemize}

A different way to calculate the class of $M$-type elements is as follows. Any such element is conjugate to a $[2\times 2]$ matrix with different eigenvalues 
\[\left(\begin{array}{cc}\lambda & 0\\0 & \mu\end{array}\right), \lambda\ne \mu.\]
However, the order of the eigenvalues is not well-defined, thus $M$ can be identified with the quotient
\[\left(\GL_2(k)/D \times \left\{\left(\begin{array}{cc}\lambda & 0\\0 & \mu\end{array}\right)| \lambda\ne \mu\right\}\right)/(\ZZ/2\ZZ)\cong M\]
where the isomorphism is given by
\[(P, \lambda, \mu)\mapsto P\left(\begin{array}{cc}\lambda & 0\\0 & \mu\end{array}\right)P^{-1}.\]
and the $\ZZ/2\ZZ$-action is given by swapping the eigenvalues $\lambda\leftrightarrow \mu$ and on $\GL_2(k)/D$ is given by multiplication by $\left(\begin{array}{cc}0 & 1\\1 & 0\end{array}\right)$ on the right. With respect to the $\ZZ/2\ZZ$-action, we have the following decomposition (see the notation of Example \ref{ex:mots2decomp}).

\begin{lemma}\label{lem:pmgl2classes}
    The virtual classes of the $+$ and $-$ parts of the varieties are as follows.
\begin{itemize}
\item[a)] $[C]:=[\GL_2(k)/D]^+=q^2$
\item[b)] $[E]:=[\GL_2(k)/D]^-=q$
\item[c)] $\left[\left\{\left(\begin{array}{cc}\lambda & 0\\0 & \mu\end{array}\right)|\lambda\ne \mu, \lambda\mu\ne 0\right\}\right]^+=q^2-2q+1$
\item[d)] $\left[\left\{\left(\begin{array}{cc}\lambda & 0\\0 & \mu\end{array}\right)|\lambda\ne \mu, \lambda\mu\ne 0\right\}\right]^-=1-q$
\item[e)] $\left[\left\{\left(\begin{array}{cc}\lambda & 0\\0 & \mu\end{array}\right)|\lambda\mu\ne 0\right\}\right]^+=q^2-q$
\item[f)] $\left[\left\{\left(\begin{array}{cc}\lambda & 0\\0 & \mu\end{array}\right)|\lambda\mu\ne 0\right\}\right]^-=1-q$.
\end{itemize}

\end{lemma}

\begin{proof}
We can identify $\GL_2(k)/D$ with the open subvariety of $\PP^1(k)\times \PP^1(k)$ that is the open complement of the diagonal subvariety. The action of $\ZZ/2\ZZ$ simply swaps the two coordinates. Therefore, we have that
\[[C]:=[\GL_2(k)/D]^+=[\left((\AA^1\cup\{\infty\})\times (\AA^1\cup \{\infty\})\right)/(\ZZ/2\ZZ)]-[\PP^1]=q^2+q+1-(q+1)=q^2.\]
This implies b) as $[\GL_2(k)/D]^++[\GL_2(k)/D)]^-=q^2+q$.

Similar stratifications can be done in the case of c), d), e) and f). Namely,
\[\left\{\left(\begin{array}{cc}\lambda & 0\\0 & \mu\end{array}\right)|\lambda\ne \mu, \lambda\mu\ne 0\right\}=(\AA^1\setminus\{0\})\times (\AA^1\setminus \{0\})\setminus (\AA^1\setminus \{0\})\]
providing
\[\left[\left\{\left(\begin{array}{cc}\lambda & 0\\0 & \mu\end{array}\right)|\lambda\ne \mu, \lambda\mu\ne 0\right\}\right]^+=q^2-q-(q-1)=q^2-2q+1.\]
This proves c) and d) as 
\[\left[\left\{\left(\begin{array}{cc}\lambda & 0\\0 & \mu\end{array}\right)|\lambda\ne \mu, \lambda\mu\ne 0\right\}\right]=(q-1)(q-2).\] 
Finally, 
\[\left\{\left(\begin{array}{cc}\lambda & 0\\0 & \mu\end{array}\right)|\lambda\mu\ne 0\right\}=(\AA^1\setminus\{0\})\times (\AA^1\setminus \{0\})\]
providing
\[\left[\left\{\left(\begin{array}{cc}\lambda & 0\\0 & \mu\end{array}\right)|\lambda\mu\ne 0\right\}\right]^+=q^2-q\]
that proves both e) and f).
\end{proof}

We note that via the quotient map, $\GL_2(k)/D$ is naturally a variety over $C=(\GL_2/D)/(\ZZ/2\ZZ)$. The lemma above enables us to compute the virtual class of $M$ in a different way.

\begin{lemma}\label{le:MoverC}
 As a variety over $C$, $[M]$ is a linear combination of the classes $[C]$ and $[E]$.
\end{lemma}

\begin{proof}
Using the motivic K\"unneth-formula (Theorem \ref{thm:motkunneth}), we have that 
\[M=\left(\GL_2(k)/D\times \left\{\left(\begin{array}{cc}\lambda & 0\\0 & \mu\end{array}\right)|\lambda\ne \mu, \lambda\mu\ne 0\right\}\right)/(\ZZ/2\ZZ)=\]
\[[\GL_2(k)/D]^+ \left\{\left(\begin{array}{cc}\lambda & 0\\0 & \mu\end{array}\right)|\lambda\ne \mu, \lambda\mu\ne 0\right\}^+ + [\GL_2(k)/D]^- \left\{\left(\begin{array}{cc}\lambda & 0\\0 & \mu\end{array}\right)|\lambda\ne \mu, \lambda\mu\ne 0\right\}^-.\]
Thus, the class of $M$ is
\[[M]=[C](q^2-2q+1)+[E](1-q)=[C](q-1)^2+[E](1-q)\]
implying our statement.
\end{proof}

In particular, we obtain that the class of $M$,
\[[M]=q^2(q-1)^2+q(1-q)=(q^2-q)(q^2-q-1).\]

Next, we parametrize the isomorphism classes of centralizer subgroups. We have the following possiblities for the centralizer subgroups of elements.
\begin{itemize}
    \item The whole group $\GL_2(k)$: this is the centralizer of the scalar matrices.
    \item $J$-type: the centralizer of an element of $J$-type. These centralizers are parametrized by $\GL_2(k)/J$ the left cosets of the subgroup $J=\left\{\left(\begin{array}{cc}\mu & b\\0 & \mu\end{array}\right)\right\}$ via conjugating the centralizer $\left\{\left(\begin{array}{cc}\lambda & b\\0 & \lambda\end{array}\right)|\lambda\ne 0\right\}$ with elements of $\GL_2(k)/J$.
    \item $M$-type: once the diagonal form is fixed of an $M$-type matrix, the centralizer subgroup is the subgroup of diagonal matrices. 
\end{itemize}

\begin{remark}\label{rem:gl2commuting}
    We note that the case of $\GL_2(k)$ is quite simple: for any tuple of commuting elements $g_1, ..., g_n$, we have an element $g$, so that $Z(g_1)\cap...\cap Z(g_n)=Z(g)$. 
\end{remark}

Next, we construct the \textit{branching matrix} that encodes how the centralizers of $n$-tuples of elements change. In the case of $\GL_2(k)$, this is fairly simple (see Remark \ref{rem:gl2commuting}), the centralizer can only change from the whole group to another centralizer.

Explicitly, we have the following.
\begin{itemize}
    \item The centralizer of an $n$-tuple of commuting elements was the whole group $\GL_2(k)$: In this case, the centralizer remains the whole group if and only if the next element is a scalar matrix. The centralizer becomes $J$-type if the next element is of $J$-type, and it becomes $M$-type if the next element is of $M$-type.
    \item The centralizer of an $n$-tuple of commuting elements was of $J$-type: In this case, the centralizer has to remain the same if we add one more commuting element. Indeed, after conjugation the centralizer is of the form $\left\{\left(\begin{array}{cc}\lambda & b\\0 & \lambda\end{array}\right)\right\}$, and thus any element of the $n$-tuple is of form $\left(\begin{array}{cc}\mu & c\\0 & \mu\end{array}\right)$ so the next elements must be as well. This provides $q(q-1)$ choices.
    \item The centralizer of an $n$-tuple of commuting elements was of $M$-type: In this case, the centralizer has to remain the same if we add one more commuting element, moreover the centralizer is conjugate to the subgroup of diagonal matrices. We parametrize the $M$-type centralizers via the motivic decomposition theorem, we follow both the $+$ and $-$ part of it. Namely, the $n+1$-st element needs to be conjugate to a diagonal matrix by the same group elements. This means that if $C_n(M)$ denotes the variety of $n$-tuples of commuting elements of $M$-type, then $C_{n+1}(M)=\left(C_n(M)\times_C \left(C\times \left\{\left(\begin{array}{cc}\lambda & 0\\0 & \mu\end{array}\right)|\lambda\mu\ne 0\right\}\right)\right)/(\ZZ/2\ZZ)$. 
\end{itemize}

Now, we focus on computing $C_{n+1}(M)$ using the recursion 
\[C_{n+1}(M)=\left(C_n(M)\times_C \left(C\times \left\{\left(\begin{array}{cc}\lambda & 0\\0 & \mu\end{array}\right)|\lambda\mu\ne 0\right\}\right)\right)/(\ZZ/2\ZZ).\]




The discussion above implies that the branching matrix of $\GL_2(k)$ is of the following form.

\begin{theorem}\label{thm:branchinggl2}
The branching matrix of $\GL_2(k)$ in the basis given by $[D]$, $[J]$, $[C]$ and $[E]$ is as follows
\[A=\left(\begin{array}{cccc}q-1 & 0 & 0 &0\\ (q-1)(q^2-1) & q(q-1) & 0 & 0\\ (q-1)^2 & 0 & q(q-1) & 1-q\\ 1-q & 0 & 1-q & q(q-1)\end{array}\right).\]
\end{theorem}

\begin{proof}
    The first columns shows how the $S$-type centralizers can change to different centralizers. The second row shows that $J$-type centralizers can only change to $J$-type centralizers. The two final columns shows how the $+$ and $-$ parts of $C_{n}(M)$ change using the e) and the f) part of Lemma \ref{lem:pmgl2classes}. 
\end{proof}

Since $[M]=(q-1)^2[C]+(1-q)[E]$ and the virtual classes are $[D]=q-1$, $[J]=(q-1)(q^2-1)$, $[C]=q^2$, $[E]=q$, we have the following.

\begin{theorem}\label{thm:commntuplesgl2}
The virtual class of the $n$-tuples of commuting elements in $\GL_2(k)$ is given by
\[vA^{n-1}w^T\]
where $w=(q-1, (q-1)(q^2-1), (q-1)^2, 1-q)$ and $v=(1,1,q^2, q)$.
\end{theorem}

This allows us to compute the classes of the $n$-tuples of commuting elements explicitly after diagonalizing the matrix $A$ which generalizes the result of \cite{sharma20}.



\begin{theorem}
    \label{thm:classofcommuting}
    We have that the virtual class of the $n$-tuples of commuting elements in $\GL_2(k)$ is
    \[\frac{1}{2}q(q^2 - 1)\left(2(q^2 - q)^{n-1} - 2(q - 1)^{n-1} + (q-1)(q^2 - 1)^{n-1} + (q - 1)^{2n-1}\right).\]
    In particular, the variety of $n$-tuples of commuting elements is of dimension $2n+2$.
\end{theorem}

\textbf{The variety of commuting $n$-tuples of $\GL_3(k)$}

We start by describing the subvariety of elements of different Jordan forms.
\begin{itemize}
    \item The scalar matrices, $S$: the closed subvariety of scalar matrices has virtual class $q-1$.
    \item A Jordan block and an eigenvector with a different eigenvalue, $J_1$: invertible matrices that are conjugate to a matrix of the form
    \[\left(\begin{array}{ccc} \lambda  & 1 &0 \\ 0 & \lambda & 0\\ 0 & 0 &\mu\end{array}\right), \lambda\ne \mu.\]
    In this case, the centralizer subgroup, $C_1$, is given by the subgroup of matrices of the form
    \[\left(\begin{array}{ccc} x  & y &0 \\ 0 & x & 0\\ 0 & 0 &z\end{array}\right).\]
    We can identify $J_1$ with the variety
    \[\GL_3(k)/C_1\times \left((\AA^1\setminus\{0\})\times (\AA^1\setminus \{0\})\setminus (\AA^1\setminus \{0\})\right)\cong J_1\]
    sending
    \[(P, \lambda, \mu)\mapsto P\left(\begin{array}{ccc} \lambda  & 1 &0 \\ 0 & \lambda & 0\\ 0 & 0 &\mu\end{array}\right)P^{-1}.\]
    Therefore
    \[[J_1]=(q-1)(q-2)[\GL_3(k)]/[C_1]=q^2(q^2-1)(q^3-1)(q-2).\]
    \item A small Jordan block, $J_2$: invertible matrices that are conjugate to a matrix of the form
    \[\left(\begin{array}{ccc} \lambda  & 1 &0 \\ 0 & \lambda & 0\\ 0 & 0 &\lambda\end{array}\right).\]
    In this case, the centralizer subgroup, $C_2$, is given by the subgroup of matrices
    \[\left(\begin{array}{ccc} u  & v &x \\ 0 & u & 0\\ 0 & z &y\end{array}\right)\]
    with $u, y\ne 0$. With a similar argument as above, we obtain that
    \[[J_2]=(q-1)[\GL_3(k)/C_2]=(q-1)(q^3-1)(q+1).\]
    \item The full Jordan block, $J_3$: invertible matrices that are conjugate to a matrix of the form
    \[\left(\begin{array}{ccc} \lambda  & 1 &0 \\ 0 & \lambda & 1\\ 0 & 0 &\lambda\end{array}\right).\]
    The centralizer of such element is the subgroup $C_3$ of the matrices
    \[\left(\begin{array}{ccc} x  & y &z \\ 0 & x & y\\ 0 & 0 &x\end{array}\right).\]
    Therefore, using the argument above, we have
    \[[J_3]=(q-1)[\GL_3(k)]/[C_3]=(q-1)(q^3-1)(q^3-q).\]
    \item Matrices with double eigenvalue, $M_1$: invertible matrices that are conjugate to a matrix of the form
    \[\left(\begin{array}{ccc} \lambda  & 0 &0 \\ 0 & \lambda & 0\\ 0 & 0 &\mu\end{array}\right), \lambda\ne \mu.\]
    In this case, the centralizer subgroup, $C_4$, is given by the subgroup of invertible matrices
    \[\left(\begin{array}{ccc} x  & y &0 \\ z & u & 0\\ 0 & 0 &v\end{array}\right).\] 
    This implies that
    \[[M_1]=(q-1)(q-2)[\GL_3(k)/C_4]=q^2(q-2)(q^3-1).\]
    \item Finally, the matrices with three distinct eigenvalues, $M_3$: invertible matrices that are conjugate to a matrix of the form
    \[\left(\begin{array}{ccc} \lambda  & 0 &0 \\ 0 & \mu & 0\\ 0 & 0 &\nu\end{array}\right), \lambda\ne \mu, \lambda\ne \nu, \mu\ne \nu.\]
    The centralizer subgroup of such element is given by the subgroup of diagonal matrices. The virtual class of $M_3$ is not $(q-1)(q-2)(q-3)[\GL_3(k)/D)]$ as the subgroup $S_3$ acts on the eigenspaces. In other words, there is no given order of $\lambda$, $\mu$ and $\nu$. 
\end{itemize}

We continue by describing $M_3$ via the Motivic Decomposition Theorem and the motivic K\"unneth-formula for $S_3$. As in Example \ref{ex:goods_nsubgroups}, we choose the good set of conjugacy classes of subgroups as the trivial subgroup $1$, the whole groups $S_3$ and the subgroup of $S_2\leq S_3$. We denote the representations of $S_3$ as follows: $T$ for the trivial representation, $\Sigma$ for the sign representation and $R$ for the 2-dimensional representation. Using these notations, we have the following statements that are statements similar to the ones in Lemma \ref{le:MoverC}.

\begin{lemma}
    We have the following motivic decompositions.
    \begin{itemize}
        \item Consider cosets of $[3\times 3]$ invertible matrices with respect to diagonal matrices, $\GL_3(k)/D$, with the $S_3$-action that permutes the columns of the matrices. Then, 
        \[[\GL_3(k)/D]_{S_3}=q^6T+q^3\Sigma+(q^5+q^4)R.\]
        \item Consider the triples of elements of $\GG_m$ with the $S_3$-action of permuting them. Then,
        \[[\GG_m^3]_{S_3}=(q^3-q^2)T+(q-1)\Sigma-q(q-1)R.\]
        \item Consider variety, $\Conf_3(\GG_m)$, of the triple of non-equal elements of $\GG_m$ with the $S_3$-action of permuting them. Then, 
        \[[\Conf_3(\GG_m)]_{S_3}=(q-1)\left((q^2-q+1)T+\Sigma-2(q-1)R\right).\]
    \end{itemize}
\end{lemma}

\begin{proof}
    The proof is very similar to the proof of Lemma \ref{lem:pmgl2classes}. We use the notation of Example \ref{ex:mots3decomp}. In order to describe the motivic decomposition
    \[[X]_{S_3}=aT+b\Sigma+cR\]
    of an $S_3$-variety $X$, we need to compute the classes of the quotients $[X]$, $[X/S_2]$ and $[X/S_3]$. In fact, we show in Example \ref{ex:mots3decomp} that
    \[a=[X/S_3],\quad b=[X]-2[X/S_2]+[X/S_3],\quad c=[X/S_2]-[X/S_3].\]


    In the case of $(\GG_m)^3$, we use Kapranov's zeta function (see Examples 1.3.2 b) with $l=[\AA^1]$ of \cite{kapranov2000elliptic}) to obtain that $[(\GG_m)^3/S_3]=q^3-q^2$. Similarly, we can use Kapranov's zeta function (again see Examples 1.3.2 b) with $l=[\AA^1]$ of \cite{kapranov2000elliptic}) and the K\"unneth-formula for motivic decompositions to get that $[(\GG_m)^2/S_2\times \GG_m]=q(q-1)^2$. Using Example \ref{ex:mots3decomp}, we obtain that
    \[[\GG_m^3]_{S_3}=(q^3-q^2)T+(q-1)\Sigma-q(q-1)R.\]
    
    In the case of $\Conf_3(\GG_m)$, we can stratify $\Conf_3(\GG_m)$ as $(\GG_m)^3\setminus Z$ where $Z$ is the subvariety where at least 2 points coincide. The subvariety $Z$ can be further stratify: exactly two points coincide or all three points coincide. This yields 
    \[[\Conf_3(\GG_m)/S_3]=q^2(q-1)-(q-1)(q-2)-(q-1)=(q-1)(q^2-q+1).\]
    Similar computation can be done to show that
    \[[\Conf_3(\GG_m)/S_2]=(q-1)(q^2-3q+3).\]
    Using Example \ref{ex:mots3decomp}, we obtain that
    \[[\Conf_3(\GG_m)]_{S_3}=(q-1)((q^2-q+1)T+\Sigma-2(q-1)R).\]

    Finally, in the case of $\GL_3/D$, we use the same trick as in Lemma \ref{le:MoverC}. We identify $\GL_3/D$ with the open subvariety of  $(\PP^2)^3$ (three points in $\PP^2$) where the three points are non-collinear. Thus, 
    \[[\GL_3/D/S_3]=[(\PP^2)^3/S_3]-[\PP^2]([(\PP^1)^3/S_3]-[\PP^1])+[\PP^2]\]
    where the first term computes the class of unordered triples in $\PP^2$, the second term computes the class of triples where two points coincide (in that case, the two points determine a unique line) and finally, the third term computes the class of triples where all points coincide. Using this stratification, we get that
    \[[\GL_3/D/S_3]=q^6.\]
    For $\GL_3/D/S_2$ (where $S_2$ permutes the first two columns), one can consider the fibration $\GL_3/D/S_2\to \Conf_2(\PP^2)$ that sends a triple of points to the configuration space of two non-equal points of $\PP^2$. This gives an $\AA^2$-fibration, and thus
    \[[\GL_3/D/S_2]=([(\PP^2)^2/S_2]-[\PP^2])[\AA^2]=q^6+q^5+q^4.\]
    Using Example \ref{ex:mots3decomp}, we obtain
    \[[\GL_3(k)/D]_{S_3}=q^6T+q^3\Sigma+(q^5+q^4)R\]
    finishing the proof.
\end{proof}

\begin{remark}
    In particular, we can compute the class of $M_2$ by taking the coefficient of $T$ in the product $[\Conf_3(\GG_m)]_{S_3}[\GL_3(k)/D]_{S_3}$ that is given as
    \[[M_2]=(q-1)(q^6(q^2-q+1)+q^3-2(q^5+q^4)(q-1))=(q^3-1)(q^6-2q^5+2q^3-q-1)-(q-1).\]
    We note that we could have obtained the virtual class of $M_2$ from the other classes:
    \[[M_2]=[\GL_3(k)]-[S]-[J_1]-[J_2]-[J_3]-[M_1]=\]
    \[=(q^3-1)\left((q^3-q)(q^3-q^2)-q^2(q^2-1)(q-2)-(q^2-1)-(q-1)(q^3-q)-q^2(q-2)\right)-(q-1)\]
    that agrees with the computation from the motivic decomposition.
\end{remark}

In the case of $\GL_3$, unline in the case of $\GL_2$, the centralizer subgroup of an Abelian group is not necessarily the centralizer of a single element. In fact, the group generated by the elements
\[\left(\begin{array}{ccc} 1 & 1 & 0\\ 0 & 1 & 0\\ 0& 0&1\end{array}\right) \mbox{ and } \left(\begin{array}{ccc} 1 & 0 & 1\\ 0 & 1 & 0\\ 0& 0&1\end{array}\right)\]
gives such an example and all examples are conjugate to this. We call such Abelian groups type $N$. Thus, we have a branching of centralizers as in the following diagram
\[
\begin{tikzcd}
    & & S\ar[rd]\ar[ld] & &\\
    & M_1\ar[ld]\ar[rd] & & J_2\ar[d]\ar[ld]\ar[rd] &\\
    M_2 & & J_1 & J_3 & N
\end{tikzcd}
\]
where the arrow denotes a possibly change in the centralizer group by adding one more commuting element.

Explicitly, we have the following theorem where we denote by $C_n(A)$ the variety of commuting $n$-tuples of type $A$. 

\begin{theorem}\label{thm:gl3recursion}
    We have the following recursion.
    \begin{itemize}
        \item $[C_{n+1}(S)]=[C_n(S)](q-1)$,
        \item $[C_{n+1}(M_1)]=[C_n(S)][M_1]+[C_n(M_1)](q-1)^2$,
        \item $[C_{n+1}(J_2)]=[C_n(S)][J_2]+[C_n(J_2)]q(q-1)$,
        \item $[C_{n+1}(J_1)]=[C_n(S)][J_1]+[C_n(J_2)](q-1)(q-2)+[C_n(M_1)](q-1)^2(q^2-1)+[C_n(J_1)]q(q-1)^2$,
        \item $[C_{n+1}(J_3)]=[C_n(S)][J_3]+[C_n(J_2)]q(q-1)^3+[C_n(J_3)]q^2(q-1)$,
        \item $[C_{n+1}(N)]=[C_n(J_2)](q-1)^2q+[C_n(N)]q^2(q-1).$
    \end{itemize}
\end{theorem}

\begin{proof}
    The statement follows from the explicit understanding on how the centralizers can change.
    
    The scalar matrices commute with all matrices. 
    
    The $M_1$-type matrices commute with matrices that are in a block form made from a $[2\times 2]$ block and a $[1\times 1]$ block. If the $[2\times 2]$ block is a scalar, then the centralizer does not change. If the $[2\times 2]$ block is a Jordan type matrix, then the centralizer changes to a $J_1$-type centralizer, and finally, if the $[2\times 2]$ block has two different eigenvalues, then the centralizer changes to an $M_2$-type centralizer.

    The $J_1$-type matrices commute with matrices that are in a block form made from a $[2\times 2]$-block with the same Jordan type (or scalar) and a $[1\times 1]$-block. Therefore, the centralizer can not change.

    The $J_2$-type matrices commute with different kind of matrices. If the eigenvalues of the matrix are different, then the centralizer changes to a $J_1$-type centralizer. If the eigenvalues are the same ($u=y$ with the notation of the description of the $J_2$-type matrices), we have three cases $xz\ne 0$, then the centralizer changes to a $J_3$-type, if $x=z=0$, then the centralizer does not change, otherwise, we get the $N$-type centralizer.

    The $J_3$-type matrices commute with matrices that are generated from the matrix (as a $k$-algebra) and thus the centralizer can not change.

    Finally, the matrices that commute with an $N$-type subgroup are in the $k$-subalgebra generated by the elements of the $N$-type subgroup, hence the centralizer of an $N$-type subgroup cannot change.
\end{proof}

Finally, we describe $C_n(M_2)$. We can get an $M_2$-type centralizer from an $S$-type, an $M_1$-type or an $M_2$-type centralizer. 
\begin{itemize}
    \item The $S$-type centralizer is the whole group.
    \item The $M_1$-type centralizer turns into an $M_2$-type centralizer if the $[2\times 2]$ block has two different eigenvalues. This means that, in this case, we have the $M$-type centralizer for the $[2\times 2]$ block (and an element for the other block). We call this $M_2'$-type.
    \item The $M_2$-type centralizer remains an $M_2$-type centralizer for any commuting matrix. We call this $M_2''$-type.
\end{itemize}

This discussion leads to the following theorem which is the $\GL_3$-version of Theorem \ref{thm:branchinggl2}.

\begin{theorem}\label{thm:branchinggl3}
    The branching matrix of $\GL_3(k)$ in the basis given by $[S]$, $[J_2]$, $[J_3]$, $[M_1]$, $[J_1]$, $[N]$; $[C\times 1]$ and $[E\times 1]$ (for $M_2'$-type), $[q^6]T$, $[q^3]\Sigma]$ and $(q^5+q^4)R$ (for the $M_2''$-type) is as follows
\[
\scalebox{0.6}{$A=\left(\begin{array}{ccccccccccc}q-1 & 0 & 0 &0 & 0& 0& 0&0&0&0&0 \\ 
(q-1)(q^3-1)(q+1) & q(q-1) & 0 & 0 &0 & 0& 0& 0&0&0&0\\
(q-1)(q^3-1)(q^3-q) & q(q-1)^3 & q^2(q-1) & 0 &0 & 0& 0& 0&0&0&0\\ 
q^2(q-2)(q^3-1) & 0 & 0 & (q-1)^2 &0 & 0& 0& 0&0&0&0\\
q^2(q^2-1)(q^3-1)(q-2) & (q-1)(q-2) & 0 & (q-1)^2(q^2-1) &q(q-1)^2 & 0& 0& 0&0&0&0\\
0 & 0 &0 &0 &q(q-1)^2 & q^2(q-1)& 0& 0&0&0&0\\
0 & 0 &0 &(q-1)^3 &0 & 0& q(q-1)^2& -(q-1)^2&0&0&0\\
0 & 0 &0 &-(q-1)^2 &0 & 0& -(q-1)^2& q(q-1)^2&0&0&0\\
(q-1)(q^2-q+1)& 0 &0 &0 &0 & 0& 0& 0& q^3-q^2 & q-1 & -q(q-1) \\
q-1& 0 &0 &0 &0 & 0& 0& 0& q-1 & q^3-q^2&-q(q-1)\\
-2(q-1)^2& 0 &0 &0 &0 & 0& 0& 0&-q(q-1)& -q(q-1) & q^3-2q^2+2q-1\\
\end{array}\right).$}\]
The virtual class of the variety of commuting $n$-tuples in $\GL_3(k)$ is given as
\[[C_n(\GL_3(k))]=vA^{n-1}w^T\]
where $v=\left(1,1,1,1,1,1,q^2, q, q^6, q^3, q^5+q^4\right)$ and 
\begin{align*}
    w=&( q-1, (q^3 - 1)(q - 1)(q + 1), (q^3 - q)(q^3 - 1)(q - 1), \\  &q^2(q^3 - 1)(q - 2), q^2(q^2 - 1)(q^3 - 1)(q - 2), 0, 0, 0, (q - 1)(q^2 - q + 1), q - 1, -2(q - 1)^2).
\end{align*}
\end{theorem}

The matrix $A$ can be diagonalized with eigenvalues: $q-1$, $q^2-q$, $q^3-1$, $q^3-q^2$ (with multiplicity 2), $(q-1)^2$, $q(q-1)^2$, $(q-1)(q^2-1)$ (with multiplicity 2) and $(q-1)^3$ (with multiplicity 2). Diagonalizing the matrix, we obtain an explicit formula.

\begin{theorem}\label{thm:commntuplegl3}
    The virtual class of the variety of commuting $n$-tuples of $\GL_3(k)$ is given by 
    \[[C_n(\GL_3(k))]=-(q-2)(q^3+2q^2+2q+1)(q^2-q)^n+\frac{1}{3}q^3(q-1)^2(q+1)(q^3-1)^n+\]
\[+\left((q+1)(q^3-1)+\frac{q^9-2q^8-q^6+3q^5-2q^4+q^3-2q^2+q-2}{q(q^2-q+1)}\right)(q^3-q^2)^n+\]
\[+\frac{q^7 - q^6 -2q^5 - q^4 + 4q^3 + 3q^2 + q - 2}{q^2 - q -1}(q-1)^n-\frac{q^3(q^3 - q^2 + 2)(q^2+q+1)}{q^2 - q + 1}(q-1)^{2n}+\]
\[-\frac{(q-2)(q^9-q^8-2q^7-q^6+2q^5+2q^4-q^3-q^2+1)}{q(q-1)(q^2-q-1)}(q(q-1)^2)^n+\frac{q^6-q^3}{2}((q-1)(q^2-1))^{n}+\]
\[+\frac{q^3(q^3 + 2q^2 + 2q + 1)}{6}(q-1)^3.\]
\end{theorem}

\begin{remark}
    As an easy consequence, we get that the dimension of $C_n(\GL_3(k))$ is $3n+6$.
\end{remark}

\begin{remark}
    The computations similar to the computations done above for $\GL_2$ and $\GL_3$ can be performed, in principle, for any group with finitely many isomorphism classes of centralizers, thus, for instance, for reductive groups \cite{garge2020finiteness}. However, the computation may be much more involved.
\end{remark}

\subsubsection{Motivic representation stability of commuting $n$-tuples }

Although it is clear from Theorem \ref{thm:classofcommuting} that the varieties of commuting $n$-tuples do not satisfy motivic stability, it motivates our conjecture that these varieties satisfy motivic representation stability. In fact, the failure of the motivic stability is a consequence of the non-trivial contribution of the non-trivial representations in the motivic representation stability.

\textbf{The case of $\GL_2$}: Before explaining motivic representation stability for a general linear reductive algebraic group, we illustrate the general proof in the case of $\GL_2$. In the previous section, we stratified the variety of commuting triples into three pieces according to the corresponding centralizer groups:
\[S^n=C_n(D)=\{A_1,...,A_n|A_i's\mbox{ are scalars}\}\]
\[C_n(J)=\left\{A_1,...A_n|\mbox{some }A_i \mbox{ is conjugate to }\left(\begin{array}{cc} \lambda & 1\\ 0 & \lambda\end{array}\right)\right\}\]
\[C_n(M)=\left\{A_1,...A_n|\mbox{some }A_i \mbox{ is conjugate to }\left(\begin{array}{cc} \lambda & 0\\ 0 & \mu\end{array}\right)\mu\ne \lambda\right\}\]
By simultaneously conjugating the elements $A_i$, we can express $C_n(J)$ as
\[C_n(J)=\Ind_H^{\GL_2} (J^n\setminus S^n)\]
where $J=\left\{\left(\begin{array}{cc}\mu & b\\0 & \mu\end{array}\right)\right\}$ is the subgroup of upper triangular matrices with equal eigenvalues and $H=J$ is the stabilizer of $J$. Similarly, 
\[C_n(M)=\Ind_K^{\GL_2}(D^n\setminus S^n)\]
where $D$ is the subgroup of diagonal matrices, $K$ being the stabilizer of $D$. Since the dimension of $S^n$ is negligible compared to $D^n$ or $J^n$, it is enough to show that 
\[\Ind_H^{\GL_2} (J^n)\quad\mbox{and}\quad \Ind_K^{\GL_2}(D^n)\]
are motivically representation stable. Since, the action of $S_n$ and $\GL_2$ commute on the commuting $n$-tuples, we have that
\[\Ind_H^{\GL_2} (J^n)/\S_{\lambda[n]}=\Ind_H^{\GL_2} (\Sym_H^{n-|\lambda|}J\times \prod_{i}\Sym_H^{\lambda_i}J)\]
and similarly
\[\Ind_K^{\GL_2} (D^n)/\S_{\lambda[n]}=\Ind_K^{\GL_2} (\Sym_K^{n-|\lambda|}D\times \prod_{i}\Sym_K^{\lambda_i}D).\]
The action of $H$ (and $K$ resp.) are linear on $J$ viewed as an open subvariety of $\AA^2$ (and on $D$ viewed as an open subvariety of $\AA^2$), thus motivic representation stability holds using Corollary \ref{cor:stabsymmetric}.

\textbf{The case of a general linear reductive algebraic group:} The strategy above can be used to show the following general statement.

\begin{theorem}\label{thm:motrepstabgengp}
    Let $G$ be a linear reductive algebraic group over an algebraically closed field $k$. Assume that all the Abelian subgroups of $G$ that are maximal under inclusion are connected.
    Then, the variety of commuting $n$-tuples in $G$ satisfies motivic representation stability.
\end{theorem}

\begin{proof}
 In a linear reductive algebraic group over an algebraically closed field $k$ has finitely many centralizers of the form $Z_G(g)$ up to conjugation \cite{garge2020finiteness}. The Abelian subgroups that are maximal under inclusion are finite intersections of centralizers of elements, therefore, we have only finitely many maximal Abelian subgroups under conjugation. 

 Now, using the same idea as in the case of $\GL_2$, we have that 
 \[\lim_{n\to \infty} \frac{[C_n(G)/S_{\lambda[n]}]}{\LL^{\dim C_n(G)}}=\lim_{n\to \infty} \sum_{A\leq G:\mbox{max Abelian}}\frac{[\Ind_{N_G(A)}^{\GL_2}(A^n)/S_{\lambda[n]}]}{\LL^{\dim C_n(G)}}.\]
 This limit exists in $\widehat{M_\LL}$ using Lemma \ref{lem:motstabformonic} and Corollary \ref{cor:resindmotivic}, because the summation is finite.
\end{proof}

We can strengthen the theorem above to deal with the case of character stacks in the case of the general linear group $\GL_r$. In other words, we show that the motivic representation stability holds not just in the Grothendieck ring of varieties, but also in the Grothendieck ring of $G$-varieties.

\begin{corollary}\label{cor:conjC}
    The $\GL_r$-character stacks of commuting $n$-tuples of $\GL_r$ satisfy motivic representation stability, i.e, the sequence $C_n(\GL_r)$. is motivically representation stable in $\K_0(\Var_k^{\GL_r})$.
\end{corollary}

\begin{proof}
     The proof for $\GL_3$ works similarly to the case of $\GL_2$ by working through all possible centralizers (see Theorem \ref{thm:gl3recursion}). 
     
     In the general case, one can use the theorem of \cite{schur1905theorie, jacobson1944schur} stating that the maximal Abelian subgroups of $\GL_r$ are open subsets of affine spaces and the action of $\GL_r$ can be extended to a linear action on the affine spaces. Therefore, Corollary \ref{cor:stabsymmetric} implies that the $\GL_r$-varieties of commuting $n$-tuples of $\GL_r$ satisfy motivic representation stability.
\end{proof}

\textbf{Counterexample}

The assumption of the connectedness in Theorem \ref{thm:motrepstabgengp} is crucial. 

\begin{proposition}
    The variety of commuting $n$-tuples in $\SL_r$ does not satisfy motivic representation stability.
\end{proposition}

\begin{proof}
    The center of $\SL_r$ is not connected, it is isomorphic to the finite cyclic group of order $r$. Therefore, the maximal Abelian subgroups are also not connected, their virtual classes are not monic polynomials in $\LL$. Since, the center commutes with any normalizer, the virtual class of the underlying variety of $\Ind_{N_G(A)}^{\SL_r}(A^n/S_{\lambda[n]})$ is also not a monic polynomial in $\LL$, and thus using Theorem \ref{thm:motrepstabgengp}, Lemma \ref{lem:motstabformonic} shows that the variety of commuting $n$-tuples in $\SL_r$ can not satisfy motivic representation stability. 
\end{proof}

Similar argument works for other linear algebraic groups as well with non-connected centers.

 \section{Motivic stability of representation varieties of free groups and free Abelian groups corresponding to reductive groups of increasing rank}

In this section, we provide an algebraic analogue of Theorem 9.6 of \cite{ramras2021homological}: we conjecture that the family of varieties of commuting $n$-tuples in $\GL_r$ satisfies motivic stability as the rank, $r$, tends to infinity. To provide evidence, we prove the conjecture for $n=2$.

\begin{conjecture}\label{con:motstabrank}
     Fix a positive integer $n$. Consider the family of varieties $C_n(\GL_r(k))$. Then,
     \[\lim_{r\to \infty} \frac{[C_n(\GL_r(k))]}{\LL^{\dim C_n(\GL_r(k))}}\]
     exists in $\widehat{M_\LL}$.

\end{conjecture}

In the case of $n=1$, the conjecture is immediate. This section is devoted to prove a quantitative version of Conjecture \ref{con:motstabrank} in the case $n=2$. 

\begin{theorem}\label{thm:motstabrank}
     Consider the family of varieties, $C_2(\GL_r(k))$. Then,
     \[\lim_{r\to \infty} \frac{[C_2(\GL_r(k))]}{q^r[\GL_r(k)]}=1\]
     in $\widehat{\K_0(\Stck_k)}$.
 \end{theorem}

The above theorem can be thought of as the motivic version of \cite{macdonald1981numbers}. We begin with analyzing the virtual classes of all the conjugacy classes of elements.

\begin{theorem}\label{thm:motstabconj}
    Consider the family of groups, $\GL_r(k)$. Then, the virtual class of the family of the union of the conjugacy classes of $\GL_r(k)$ is motivically stable.
\end{theorem}

\begin{proof}
    The proof is a straightforward generalization of the proof in \cite{macdonald1981numbers}, we add the details for the sake of completeness. Consider, $\GL_r$, and possible characteristic polynomial of an $[r\times r]$ matrix
    \[p(x)=\prod_{i=1}^k (x-\lambda_i)^{a_i}\]
    where the matrix has $k$ Jordan blocks of sizes $a_1\geq a_2\geq ...\geq a_k$ with eigenvalues $\lambda_i$ respectively. Consider the alteration of $p(x)$,
    \[q(x)=\prod_{i=1}^k (1-\lambda_i x)^{a_i}.\]
    We parametrize the different classes of conjugacy classes via factoring $q(x)$ using partitions corresponding to the size of the Jordan blocks:
    $q(x)=\prod_{l=1}^r h_l(x)^l$
    where $h_l(x)$ is the product of the factors of $q(x)$ of the form $(1-\lambda_i x)$ where $a_i=l$. The polynomials, $h_l(x)$, are polynomials with constant term 1 and of degree $n_l$ satisfying $\sum_{l=1}^r  ln_l=r$. Thus, the polynomials, $h_l(x)$ are parametrized by $\AA^{n_l-1}\times (\AA^1\setminus \{0\})$. Furthermore, the polynomials $h_l(x)$ uniquely determine Jordan block, therefore, the virtual class of the families of conjugacy classes corresponding to the partition $(a_1,a_2,...)$ is given as $\prod_{l=1}^r (q^{n_i}-q^{{n_i}-1})$. Using the same calculation as in \cite{macdonald1981numbers}, we get that the virtual class of all conjugacy classes, $[\mathcal{C}_r]$, satisfies
    \[\lim_{r\to \infty} \frac{[\mathcal{C}_r]}{q^r}=1\]
    in $\widehat{M_\LL^G}$, proving the statement.
\end{proof}

\begin{remark}
    For sake of clarity, we illustrate the computation for $\GL_2(k)$. The Jordan forms come in three different forms.
    \begin{itemize}
        \item The scalar matrices: The virtual class of the family of conjugacy classes of matrices that are conjugate to scalar matrices is $q-1$.
        \item The $J$-type matrices: The virtual class of the family of conjugacy classes of matrices that are conjugate to $J$-type matrices is $(q-1)$, namely that is given by the unique eigenvalue. In other words, $[J/\GL_2(k)]=q-1$.
        \item The $M$-type matrices: The virtual class of the family of conjugacy classes of matrices that are conjugate to a diagonal matrix of two different eigenvalues can be computed as follows. We parametrize the family using the characteristic polynomial. The variety of all characteristic polynomials is isomorphic to $\AA^1\times (\AA^1\setminus \{0\})$, parametrizing the two coefficients (the constant term cannot be 0). The $M$-type matrices correspond to characteristic polynomials with two different roots in this case. The scalar type matrices correspond to the characteristic polynomials that have a double root, meaning that the virtual class of the family of conjugacy classes of $M$-type matrices is $q(q-1)-(q-1)=(q-1)^2$.
    \end{itemize}
    So, the virtual class of all conjugacy classes of elements of $\GL_2(k)$ is $q-1+q-1+(q-1)^2=q^2-1$.
\end{remark}

Now, we turn our attention to the variety of commuting pairs, $C_2(\GL_r(k))$. We begin by describing the pieces of $C_2(\GL_r(k))$ as quotients by some subgroups of the symmetric group $S_r$. Consider a family of conjugacy classes given by a fixed type of Jordan normal form, $\mathcal{J}$. In other words, we consider the subvariety $C_2^\mathcal{J}(\GL_r(k))$ consisting of those commuting pairs $(A,B)$ where $A$ is conjugate to a matrix with a Jordan normal form of type $\mathcal{J}$. The possible such Jordan normal forms form a variety that is the quotient of the space of possible eigenvalues modulo a group action that permutes the possible eigenvalues. We denote this quotient as $E_\mathcal{J}/H_\mathcal{J}$ where $E_\mathcal{J}$ denotes the space of possible eigenvalues and $H_\mathcal{J}$ denotes the group acting on it. Note that the group $H_\mathcal{J}$ is a product of symmetric groups, and thus it has a "good" set of conjugacy classes of subgroups.

Using this, we get an algebraic map $\pi: C_2^\mathcal{J}(\GL_r(k))\to E_\mathcal{J}/H_\mathcal{J}$ assigning to a pair $(A,B)$ of commuting matrices the conjugacy class of $A$.

For example, consider the Jordan normal form
\[\left(\begin{array}{cccc} \lambda & 1 & 0 & 0\\
0 & \lambda & 0 & 0\\
0 & 0 & \mu & 1\\
0 & 0 & 0 & \mu\\
\end{array}\right)
\]
with $\lambda\ne \mu$, then the corresponding variety is given as the quotient of $(\AA^1\setminus \{0\})\times (\AA^1\setminus \{0\})\setminus \Delta_{\AA^1\setminus \{0\}}$ by the group $\ZZ/2\ZZ$ swapping the two coordinates (corresponding to the values of $\lambda$ and $\mu$).

Consider the Cartesian product

\[\begin{tikzcd}
    {X_\mathcal{J}}\ar[r] \ar[d]& C_2^{\mathcal{J}}(\GL_r(k))\ar[d,"\pi"]\\
 E_\mathcal{J} \ar[r] & E_\mathcal{J}/H_\mathcal{J}
\end{tikzcd}
\]

where the bottom horizontal map is the quotient map $E_\mathcal{J}\to E_\mathcal{J}/H_\mathcal{J}$. The Cartesian product $X_\mathcal{J}$ is the variety of pairs $(A,B)$ where $A$ is conjugate to a matrix with Jordan type $\mathcal{J}$ and $A$ and $B$ commute. The representatives of the Jordan normal forms of the same type have the same centralizers (denoted by $Z_\mathcal{J}$), and therefore, $X_\mathcal{J}$ is isomorphic to $E_\mathcal{J}\times (\GL_r(k)\times Z_\mathcal{J}/Z_\mathcal{J})$ where the action of $Z_\mathcal{J}$ on $\GL_r(k)\times Z_\mathcal{J}$ is given as
\[t.(g,z)=(gt,t^{-1}zg).\]
Indeed, the map 
\[E_\mathcal{J}\times (\GL_r(k)\times Z_\mathcal{J})\to X_\mathcal{J}\]
given by
\[(A,g,z)\mapsto (gAg^{-1}, gzg^{-1})\]
is surjective, and the elements $(A,g,z)$ are mapped to the same element as the elements $t.(A,g,z)=(A,gt, t^{-1}zt)$ for $t\in Z_\mathcal{J}$.

The action of $H_\mathcal{J}$ on $E_\mathcal{J}$ lifts to an action on $X_\mathcal{J}$, and thus $C_2^{\mathcal{J}}(\GL_r(k))$ is the quotient of $E_\mathcal{J}\times (\GL_r(k)\times Z_\mathcal{J}/Z_\mathcal{J})$ with the group $H_\mathcal{J}$ where $H_\mathcal{J}$ acts as 
\[h.(A,g,z):=(PAP^{-1}, gP, P^{-1}zP)\]
where $P$ is a permutation matrix corresponding to the element $h\in H_\mathcal{J}$.  Lifting the action to $E_\mathcal{J}\times \GL_r(k)\times Z_\mathcal{J}$, we obtain a Cartesian diagram
\[\begin{tikzcd}
    E_\mathcal{J}\times \GL_r(k)\times Z_\mathcal{J}\ar[d]\ar[r] & (E_\mathcal{J}\times \GL_r(k)\times Z_\mathcal{J})/H_\mathcal{J}\ar[d]\\
    E_\mathcal{J}\times (\GL_r(k)\times Z_\mathcal{J}/Z_\mathcal{J})\ar[r] \ar[d]& C_2^{\mathcal{J}}(\GL_r(k))\ar[d,"\pi"]\\
 E_\mathcal{J} \ar[r] & E_\mathcal{J}/H_\mathcal{J}.
\end{tikzcd}
\]
The map $E_\mathcal{J}\times \GL_r(k)\times Z_\mathcal{J}\to E_J$ is a $\GL_r(k)\times Z_\mathcal{J}$-torsor. Since, the group $\GL_r(k)$ and $Z_\mathcal{J}$ are special linear algebraic groups in the sense of \cite{chevalley1958anneaux} ($Z_\mathcal{J}$ is an extension of a smooth unipotent group by products of general linear groups), therefore, \[[E_\mathcal{J}\times \GL_r(k)\times Z_\mathcal{J}/H_\mathcal{J}]=[E_\mathcal{J}/H_\mathcal{J}][\GL_r(k)][Z_\mathcal{J}].\] Similarly, the map 
\[E_\mathcal{J}\times \GL_r(k)\times Z_\mathcal{J}\to  E_\mathcal{J}\times (\GL_r(k)\times Z_\mathcal{J}/Z_\mathcal{J})\]
is a $Z_\mathcal{J}$-torsor, thus 
\[[E_\mathcal{J}\times \GL_r(k)\times Z_\mathcal{J}/H_\mathcal{J}]=[C_2^{\mathcal{J}}(\GL_r(k))][Z_\mathcal{J}].\]
This implies that $[C_2^{\mathcal{J}}(\GL_r(k)]=[E_\mathcal{J}/H_\mathcal{J}][\GL_r(k)]$ in $\K_0(\Stck_k)$.

Now, we are ready to prove Theorem \ref{thm:motstabrank}.

\begin{proof}[Proof of Theorem \ref{thm:motstabrank}]
    The above implies that $[C_2(\GL_r(k))]=[\GL_r(k)]\sum_\mathcal{J} [E_\mathcal{J}/H_\mathcal{J}]=[\GL_r(k)][\mathcal{C}_r]$. Using Theorem \ref{thm:motstabconj}, we obtain
    \[\lim_{r\to \infty}\frac{[C_2(\GL_r(k))]}{[\GL_r(k)]q^r}=1\]
    in $\widehat{\K_0(\Stck_k)}$ proving our theorem.
\end{proof}

\begin{remark}
    In the proof of Theorem \ref{thm:motstabrank} we relied on two key statements: a) the virtual class of all conjugacy classes of $\GL_r(k)$ is motivically stable, b) the groups $\GL_r(k)$ and centralizer subgroups of the Jordan blocks of $\GL_r(k)$ are special algebraic groups in the sense of \cite{chevalley1958anneaux}. These facts also hold for the family of $\SL_r(k)$. In fact, a straight adaptation of the proof in \cite{macdonald1981numbers} on the number of conjugacy classes of $\SL_r(k)$ shows that the virtual class of the space of all conjugacy classes, $[\mathcal{C}(\SL_2(k))]$ of $\SL_r(k)$ is motivically stable with
\[\lim_{r\to \infty}\frac{[\mathcal{C}(\SL_r(k))]}{q^r}=\frac{1}{q-1}.\]
    Thus, the following theorem can be deduced verbatim.
\end{remark}

\begin{theorem}
    Consider the family of groups, $\SL_r(k)$. Then, 
    \[\lim_{r\to \infty}\frac{[C_2(\SL_r(k))]}{q^r[\SL_r(k)]}=\frac{1}{q-1}\]
    in $\widehat{\K_0(\Stck_k)}$.
\end{theorem}


\bibliographystyle{abbrv}
\addcontentsline{toc}{section}{References}
\bibliography{bibliography}

\end{document}